\documentclass{amsart}
\newcommand{\Z}{\mathbb{Z}}
\newcommand{\Q}{\mathbb{Q}}
\newcommand{\R}{\mathbb{R}}
\newcommand{\del}{\!\setminus\!}
\newcommand{\one}{{\bf 1}}
\newcommand{\BB}{\mathcal{B}}
\newcommand{\ignore}[1]{}

\newcommand{\supp}{\text{supp}}
\newcommand{\dom}{\text{dom}}
\newcommand{\cha}{\operatorname{char}}
\newcommand{\im}{\operatorname{im}}

\newcommand{\Rinf}{\mathbb{R}_\infty}
\newcommand{\Zinf}{\mathbb{Z}_\infty}
\newcommand{\val}{\mathrm{val}}

\newtheorem{theorem}{Theorem}
\newtheorem{definition}[theorem]{Definition}
\newtheorem{lemma}[theorem]{Lemma}
\newtheorem{corollary}[theorem]{Corollary}
\newtheorem{example}[theorem]{Example}

\usepackage{fullpage}
\usepackage{graphicx}
\usepackage{amssymb}
\usepackage{microtype}
\usepackage[utf8]{inputenc}
\usepackage{a4wide}

\begin{document}
\title{Algebraic matroids and Frobenius  flocks}
\author{Guus P.~Bollen}\thanks{GPB is funded by a grant from Stichting Computer Algebra Nederland}
\author{Jan Draisma}
\thanks{JD is partially funded by Vidi
and Vici grants from the Netherlands Organisation for Scientific Research (NWO)}
\author{Rudi Pendavingh}

\address[Guus P.~Bollen]{Department of Mathematics and Computer Science,
Eindhoven University of Technology, P.O.~Box 513, 5600 MB Eindhoven, The Netherlands}
\email{g.p.bollen@tue.nl}

\address[Jan Draisma]{Mathematical Institute, University of Bern,
Sidlerstrasse 5, 3012 Bern, Switzerland; and Eindhoven University of
Technology, The Netherlands}
\email{jan.draisma@math.unibe.ch}

\address[Rudi Pendavingh]{Department of Mathematics and Computer Science,
Eindhoven University of Technology, P.O.~Box 513, 5600 MB Eindhoven, The Netherlands}
\email{r.a.pendavingh@tue.nl}

\begin{abstract}
We show that each algebraic representation of a matroid $M$ in positive characteristic determines a matroid valuation of $M$, which we have named the {\em Lindstr\"om valuation}. If this valuation is  trivial, then a linear representation of $M$ in characteristic $p$ can be derived from the algebraic representation. Thus, so-called rigid matroids, which only admit  trivial valuations, are algebraic in positive characteristic $p$ if and only if they are linear in characteristic $p$.

To construct the Lindstr\"om valuation, we introduce new matroid representations called flocks, and show that each algebraic representation of a matroid induces  flock representations.
\end{abstract}

\maketitle

\section{Introduction}
Several years before Whitney defined matroids in \cite{Whitney1935}, Van der Waerden \cite{vanderWaerden1940} recognized the commonality between algebraic and linear independence which is the defining property of matroids  (see \cite[Section 39.10b]{SchrijverBook}).
Today, linear matroids are the subject of some of the
deepest theorems in combinatorics. By contrast, our
understanding of {\em algebraic} matroids is much less advanced.

Algebraic matroids seem forebodingly hard to deal with,
compared to  linear matroids. Deciding if a matroid has a
linear representation over a given field amounts to solving
a system of polynomial equations over that field, which
makes the linear representation problem decidable over
algebraically closed fields
(for the situation of representability over $\Q$, see
\cite{Sturmfels87}). We do not know of any procedure to decide if a matroid has an algebraic representation; for example, it is an open problem if the Tic-Tac-Toe matroid, a matroid of rank 5 on 9 elements, is algebraic \cite{Hochstattler1997}. It is possible to describe the set of all linear representations of a given matroid as being derived from a `universal' matroid representation of that matroid over a partial field \cite{PendavinghZwam2010}. We tried to find a full set of invariants to classify (equivalence classes of) algebraic representations of the tiny uniform matroid $U_{2,3}$, and failed. Since algebraic representations appear so difficult to deal with in full detail, we are pleased to report that we found a nontrivial invariant of algebraic matroid representations in positive characteristic,  taking inspiration in particular from the work of Lindstr\"om \cite{Lindstrom1985}.

Throughout the paper, we use the language of algebraic
geometry to describe matroid representations, and we work
over an algebraically closed field $K$. An irreducible algebraic variety $X\subseteq K^E$ determines a matroid $M(X)$ with ground set $E$ by declaring a set $I\subseteq E$ independent if the projection of $X$ on $K^I$ is dominant, that is, if the closure of  $\{x_I: x\in X\}$ in the Zariski topology equals $K^I$. In the special case that $X$ is a linear space, $M(X)$ is exactly the matroid represented by the columns of any matrix whose rows span $X$. We next translate results of Ingleton and Lindstr\"om to the language of algebraic geometry.

Ingleton argued that if $K$ has characteristic 0, then for
any sufficiently general point $x\in X$, the tangent space
$T_xX$ of $X$ at $x$ will have the same dominant projections
as $X$ itself, so that then $M(X)=M(T_x X)$. Since such a
sufficiently general point always exists --- it suffices
that $x$ avoids finitely many hypersurfaces in $X$ --- a matroid which has an algebraic representation over $K$ also has a linear representation over $K$, in the form of a tangent space $T_x X$.

Ingletons argument does not generalize to fields $K$ of positive characteristic $p$. Consider the variety $X=V(x_1-x_2^p)=\{x\in K^2: x_1-x_2^p=0\}$, which represents the matroid on $E=\{1,2\}$ with independent sets $\emptyset, \{1\}, \{2\}$. The tangent space of $X$ at any $x\in X$ is $T_x X=V(x_1)$, which represents a matroid in which $\{1\}$ is a dependent set. Thus, $M(X)\neq M(T_xX)$ for all $x\in X$.

Lindstr\"om demonstrated that for some varieties $X$ this
obstacle may be overcome by applying the Frobenius map $F:
x\mapsto x^p$ to some of the coordinates, to derive varieties from $X$ which represent the same matroid. In case of the counterexample $X=V(x_1-x_2^p)$ above, we could define  $X'=\{(x_1, F(x_2)): x\in X\}.$
Then $M(X')=M(X)$, and $X'=V(x_1-x_2)$, so that $M(X')=M(T_{x'}X')$ for some (in fact, all) $x'\in X'$.
In general for an $X\subseteq K^E$, if we fix a vector $\alpha\in \Z^E$,  put $$\alpha x:=(F^{-\alpha_i}x_i)_{i\in E}\text{ and }\alpha X:=\{ \alpha x: x\in X\},$$ then it can be argued that $M(\alpha X)=M(X)$. This gives additional options for finding a suitable tangent space.

Lindstr\"om showed in \cite{Lindstrom1985} that if $X$ is any algebraic representation of the Fano matroid, then there necessarily exists an $\alpha$ so that $M(X)=M(T_\xi \alpha X)$ for a sufficiently general $\xi\in \alpha X$. Thus any algebraic representation of the Fano matroid spawns a linear representation in the same characteristic.

The choice of the matroid in Lindstr\"om's argument is not arbitrary. If we consider an algebraic representation $X$ of the non-Fano matroid in characteristic 2, then there cannot be an $\alpha\in\Z^E$ so that in a sufficiently general point $\xi\in \alpha X$ we have $M(X)=M(T_\xi \alpha X)$: then the tangent space would be a linear representation of the non-Fano matroid in characteristic 2, which is not possible. Thus any attempt to balance out the Frobenius actions $\alpha\in \Z^E$ will be fruitless in this case.

In the present paper, we consider the overall structure of
the map $\alpha\mapsto M(T_{\xi_\alpha} \alpha X)$, where
for each $\alpha$, $\xi_\alpha$ is the generic point of
$\alpha X$. A central result of this paper is that
a sufficiently general $x\in X$
satisfies $M(T_{\alpha x} \alpha X)=M(T_{\xi_\alpha} \alpha
X)$ for all $\alpha \in \Z^E$ (Theorem \ref{cor:AlgFlock}). This is nontrivial, as now
$x$ must {\em a priori} avoid countably many hypersurfaces, which
one might think could cover all of $X$.

Fixing such a general $x$, the assignment $V:\alpha\mapsto V_\alpha:=T_{\alpha x} \alpha X$ is what we have named a {\em Frobenius flock}: each $V_\alpha$ is a linear subspace of $K^E$ of fixed dimension $\dim V_\alpha=\dim X =:d$, and moreover
\begin{itemize}
\item[(FF1)] $V_{\alpha}/ i=V_{\alpha+e_i}\del i$ for all $\alpha\in\Z^E$ and $i\in E$; and
\item[(FF2)] $V_{\alpha+\one}=\one V_{\alpha} $ for all $\alpha\in \Z^E$.
\end{itemize}
Here we wrote $W/i:=\{w_{E-i}: w \in W, w_i=0\}$ and $W\del i:=\{w_{E-i}: w \in W\}$; $\one$ is the all-one vector and $e_i$ is the $i$-th unit vector, with $i$-th coordinate 1 and 0 elsewhere.

In this paper, we are primarily concerned with properties of the matroids represented by the vector spaces $V_\alpha$ --- in particular, we want to understand when there exists an $\alpha\in \Z^E$ so that  $M(X)=M(V_\alpha)$. We will call an assignment $M:\alpha\mapsto M_\alpha$ of a matroid of rank $d$ on $E$ to each $\alpha\in \Z^E$ a {\em matroid flock} if it satisfies
\begin{itemize}
\item[(MF1)] $M_\alpha/i=M_{\alpha+e_i}\del i$ for all $\alpha\in \Z^E$ and $i\in E$; and
\item[(MF2)] $M_\alpha=M_{\alpha+\one}$ for all $\alpha\in \Z^E$.
\end{itemize}
Clearly, any Frobenius flock $V$ gives rise to a matroid flock by taking $M_\alpha=M(V_\alpha)$. In particular, each algebraic matroid representation $X\subseteq K^E$ gives rise to a matroid flock
$M:\alpha\mapsto M(T_{\alpha x} \alpha X)$. This matroid flock does not depend on the choice of the very general point $x\in X$, since $M(T_{\alpha x} \alpha X)=M(T_{\alpha x'} \alpha X)$ for any two such very general points.

The matroid flock axioms (MF1), (MF2) may not appear to be very restrictive, but
matroid flocks (and Frobenius flocks) are remarkably
structured. For instance, while a matroid
flock is {\em a priori} specified by an infinite amount of
data, it follows from our work that a finite description
always exists. More precisely, we show that for each matroid flock  $M$ there is a matroid valuation $\nu:\binom{E}{d}\rightarrow \Z\cup\{\infty\}$ in the sense of Dress and Wenzel \cite{DressWenzel1992}, which conversely determines each $M_\alpha$ as follows: $B$ is a basis of $M_\alpha$ if and only if
$$\sum_{i\in B}\alpha_i-\nu(B)\geq \sum_{i\in B'}\alpha_i-\nu(B')$$
for all $B'\in \binom{E}{d}$ (Theorem \ref{thm:mf_char}).

In case the matroid flock $M$ derives from an algebraic representation $X$, this valuation $\nu$ gives us the overview over the matroids $M(T_{\alpha x} \alpha X)$ we sought.  In particular, we have $M(X)=M_\alpha=M(T_{\alpha x}\alpha X)$ if and only if  $\nu(B)=\sum_{i\in B}\alpha_i$ for all bases $B$ of $M(X)$, i.e. if $\nu$ is a {\em trivial} valuation. Recognizing the seminal work of Bernt Lindstr\"om, we named $\nu$ the {\em Lindstr\"om valuation} of $X$.

Dress and Wenzel call a matroid {\em rigid} if it  supports  trivial valuations only; and so we obtain a main result of this paper, that if a matroid $M$ is rigid, then $M$ is algebraic in characteristic $p$ if and only if $M$ is linear in characteristic $p$ (Theorem \ref{thm:alg_lin}).

In summary, we find that an irreducible variety $X\subseteq K^E$ not only determines a matroid $M(X)$, but also a valuation of that matroid. This reminds of the way a matrix $A$ over a valuated field $(K, \mbox{val})$ not only determines a linear matroid $M(A)$, but also a  valuation $\nu:B\mapsto \mbox{val}(\det(A_B))$ of $M(A)$. We will demonstrate that the valuated matroids obtained from algebraic representations in characteristic $p$  include all the valuated matroids obtained from linear representations over the valuated field $(\mathbb{Q},\val_p)$.

It is an open problem to characterize the valuated matroids which may arise from algebraic representations in characteristic $p$.

\subsection*{Acknowledgement}
We  thank Josephine Yu for identifying the lifting functions associated with alcoved polytopes. Her answer to our query led us to conjecture what became Theorem \ref{thm:mf_char} below.

\section{Preliminaries}
\subsection{Matroids}
A {\em matroid} is a pair $(E, \BB)$ where $E$ is a finite set and $\BB\subseteq 2^E$
is a set satisfying
\begin{itemize}
\item[(B1)] $\BB\neq \emptyset$; and
\item[(B2)] for all $B,B' \in \BB$ and $i \in B \setminus B'$, there exists a $j \in B' \setminus B$ so that  both $B-i+j\in \BB$ and $B'+i-j\in \BB$.
    \end{itemize}
Here we used the shorthand $B-i+j:=(B\cup\{j\})\del\{i\}$. The set $E$ is the {\em ground set} of the matroid $M=(E, \BB)$, and we call the elements $B\in \BB$ the {\em bases} of $M$. We refer to Oxley \cite{OxleyBook} for standard definitions and results in matroid theory. We summarize a few  notions which are used in this paper below.

If $M=(E, \BB)$ is a matroid, then a set $I\subseteq E$ is {\em independent} if $I\subseteq B$ for some basis $B$ of $M$. The independent sets of a matroid satisfy the following {\em independence axioms}:
\begin{itemize}
\item[(I1)] $\emptyset$ is independent;
\item[(I2)] if $I$ is independent and $J\subseteq I$, then $J$ is independent; and
\item[(I3)] if $I$ and $J$ are independent and $|I|<|J|$, then there exists a $j\in J\setminus I$ so that $I+j$ is independent.
\end{itemize}
A basis of a matroid may be characterized as an inclusionwise maximal independent set, and the independence axioms (I1), (I2), (I3) then imply the basis axioms (B1) and (B2). The independence axioms thus constitute an alternative definition of matroids.

The {\em rank function} of a matroid $M=(E, \BB)$ is the function $r_M:2^E\rightarrow \mathbb{N}$ determined by
$$r_M(J):=\max\{|J\cap B|: B\in \BB\}$$
for each $J\subseteq E$. We have $B\in \BB$ if and only if $|B|=r_M(B)=r_M(E)$, and so $(E,r_M)$ fully determines $M$. We will write $r(M):=r_M(E)$ for the rank of the entire matroid. Given $I\subseteq E$, $M\del I$ is the matroid on ground set $E\del I$ with rank function $r_{M\del I}:J\mapsto r_M(J)$, and $M/I$ is the matroid on $E\del I$ with rank function $r_{M/I}:J\mapsto r_M(I\cup J)-r_M(I)$.

The {\em connectivity function} of a matroid $M=(E, \mathcal{I})$ is the function $\lambda_M:2^E\rightarrow \mathbb{N}$
determined by
$$\lambda_M(J):=r_M(J)+r_M(\overline{J})-r_M(E)$$
for each $J\subseteq E$. Here and henceforth we denote $\overline{J}:=E\del J$, where $E$ is the ground set.

\subsection{Linear spaces and linear matroids}
Let $K$ be a field and let $E$ be a finite set. If $w\in K^E$, and $I\subseteq E$, then $w_I\in K^I$ is the restriction of $w$ to $I$. For a linear subspace $W\subseteq K^E$ and a set $I\subseteq E$, we define
$$W\del I:=\{w_{\overline{I}}: w\in W\}\text{ and }W/I:=\{w_{\overline{I}}: w\in W, w_I=0\}$$
Let $r_W:2^E\rightarrow \mathbb{N}$ be the function determined by $r_W(I):=\dim(W\del\overline{I})$. The following is well known.
\begin{lemma}Let $K$ be a field, let $E$ be a finite set, and let  $W\subseteq K^E$ be a linear subspace. Then $r_W$ is the rank function of a matroid on $E$.
\end{lemma}
We write $M(W)$ for the matroid on $E$ with rank function $r_W$. Taking minors of $M(W)$ commutes with taking `minors' of  $W$: for each $I\subseteq E$, we have $M(W)\del I=M(W\del I)$, and $M(W)/I=M(W/I)$.

If $M=M(W)$, then $\lambda_M(I)=0$ exactly if
$$\dim W=r_M(E)=r_M(\overline{I})+r_M(I)= \dim(W\del I) + \dim (W\del \overline{I}).$$
This is equivalent with $W=(W\del I)\times (W\del \overline{I})$.

\subsection{Matroid valuations} For a finite set $E$ and a natural number $d$, we write $$\binom{E}{d}:=\{J\subseteq E: |J|=d\},$$ and we will abbreviate $\Rinf:=\R\cup\{\infty\}$ and $\Zinf:=\Z\cup\{\infty\}$.
A \emph{matroid valuation} is a function  $\nu: \binom{E}{d} \rightarrow \Rinf$, such that
\begin{itemize}
\item[(V1)] there is a $B\in \binom{E}{d}$ so that $\nu(B)<\infty$; and
\item[(V2)] for all $B,B' \in \binom{E}{d}$ and $i \in B \setminus B'$, there exists a $j \in B' \setminus B$ so that $$\nu(B) + \nu(B') \geq \nu(B-i+j) + \nu(B'+i-j).$$
\end{itemize}
Matroid valuations were introduced by Dress and Wenzel in \cite{DressWenzel1992}. For any valuation $\nu$, the set $$\BB^\nu:=\left\{B\in \binom{E}{d}: \nu(B)<\infty\right\}$$  satisfies the matroid basis axioms (B1) and (B2), so that $M^{\nu}:=(E,\BB^\nu)$ is a matroid, the {\em support matroid} of $\nu$. If $N=M^\nu$, then we also say that $\nu$ is a valuation of $N$.

\subsection{Murota's discrete duality theory}
We briefly review the definitions and results we use from  \cite{MurotaBook}.
For an $x\in \Z^n$, let
$$\supp(x):=\{i: x_i\neq 0\}, ~ \supp^+(x):=\{i: x_i> 0\},~\supp^-(x):=\{i: x_i< 0\}.$$
For any function $f:\Z^n\rightarrow \Rinf$, we write $\dom(f):=\{x\in \Z^n: f(x)\in\R\}$.
We write $e_i$ for the $i$-th unit vector, with $i$-th coordinate 1 and 0 elsewhere.
\begin{definition}
A function $f: \Z^n \rightarrow \Zinf$ is called \emph{M-convex} if:
\begin{enumerate}
  \item $\dom(f)\neq\emptyset$; and
  \item for all $x,y \in \dom(f)$ and $i \in \supp^+(x-y)$, there exists $j \in \supp^-(x-y)$ so that $f(x) + f(y) \geq f(x - e_i + e_j) + f(y + e_i - e_j)$.
\end{enumerate}
\end{definition}

Let $x,y \in \Z^n$. We write $x\vee y := (\max\{x_i,y_i\})_{i }$ and  $x \wedge y := (\min\{x_i,y_i\})_{i}$, and let $\one$ denote the all-one vector.

\begin{definition}A function $g: \Z^n\rightarrow \Zinf$ is  \emph{L-convex} if
\begin{enumerate}
  \item dom$(g) \neq \emptyset$;
  \item  $g(x) + g(y) \geq g(x \vee y) + g(x \wedge y)$ for all $x,y \in \Z^n$; and
  \item there exists an $r \in \Z$ so that $g(x + \one) = g(x) + r$ for all $x \in \Z^n$.
\end{enumerate}
\end{definition}
For any $h:\Z^n\rightarrow \Zinf$ with nonempty domain, the \emph{Legendre-Fenchel dual} is the function $h^\bullet: \Z^n \rightarrow \Zinf$ defined by
$$h^\bullet(x) = \sup \{ x^T y - h(y) : y \in \Z^n \}.$$
The following key theorem describes the duality of M-convex and L-convex functions \cite{MurotaBook}.
\begin{theorem}\label{thm:MLdual} Let $f,g:\Z^n\rightarrow \Zinf$. The following are equivalent.
\begin{enumerate}
\item $f$ is M-convex, and $g=f^\bullet$; and
\item $g$ is L-convex, and $f=g^\bullet$.
\end{enumerate}
\end{theorem}
Finally, Murota provides the following local optimality criterion for L-convex functions. Let $e_I:=\sum_{i\in I} e_i$ for any $I\subseteq \{1,\ldots, n\}$, and let $\one$ denote the all-one vector.
\begin{lemma}\label{thm:Llocalopt}(Murota \cite{MurotaBook}, 7.14)
Let $G$ be an L-convex function on $\Z^n$, and let $x \in \Z^n$. Then
\[\forall y \in \Z^n: G(x) \leq G(y) \Longleftrightarrow \left\{
                                                                    \begin{array}{ll}
                                                                      \forall I \subset \{1,\ldots,n\}: G(x) \leq G(x + e_I), \\
                                                                      G(x) = G(x + \one).
                                                                    \end{array}
                                                                  \right.
\]
\end{lemma}

\section{Matroid flocks}
A {\em matroid flock of rank $d$ on $E$} is a map $M$ which assigns a matroid $M_\alpha$ on $E$ of rank $d$ to each $\alpha\in \Z^E$, satisfying the following two axioms.
\begin{itemize}
\item[(MF1)] $M_\alpha/ i=M_{\alpha+e_i}\del i$ for all $\alpha\in\Z^E$ and $i\in E$.
\item[(MF2)] $M_\alpha=M_{\alpha+\one}$ for all $\alpha\in \Z^E$.
\end{itemize}

\begin{example}
Let $E=\{1,2\}$ and let $M_0$ be the matroid on $E$ with bases $\{1\}$ and
$\{2\}$. We claim that this extends in a unique manner to a matroid flock
$M$ on $E$ of rank $1$. Indeed, by (MF1) the rank-zero matroid $M_0/1$ equals $M_{e_1} \del 1$, so
that $\{1\}$ is the only basis in $M_{e_1}$. Repeating this argument, we
find $M_{ke_1}=M_{e_1}$ for all $k>0$. Similarly, $M_{ke_2}$
is the matroid with only one basis $\{2\}$ for all $k>0$.  Using (MF2)
this determines $M_\alpha$ for all $\alpha=(k,l)$: it equals $M_0$ if $k=l$,
$M_{e_1}$ if $k>l$, and $M_{e_2}$ if $k<l$. The phenomenon that most
$M_\alpha$ follow from a small number of them, like migrating birds
follow a small number of leaders, inspired our term {\em
matroid flock}. \hfill $\clubsuit$
\end{example}

For a matroid valuation $\nu:\binom{E}{d}\rightarrow \Rinf$ and an $\alpha\in \R^E$, let
$$\BB^\nu_\alpha:=\left\{B\in \binom{E}{d}: e_B^T\alpha-\nu(B)=g^\nu(\alpha)\right\},\text{ where }g^\nu(\alpha)=\sup\left\{e_{B'}^T\alpha-\nu(B'): B'\in \binom{E}{d}\right\}$$
and put $M^\nu_\alpha:=(E, \BB^\nu_\alpha)$.

The main result of this section will be the following  characterization of matroid flocks.
\begin{theorem} \label{thm:mf_char}Let $E$ be a finite set, let $d\in \mathbb{N}$, and let $M_\alpha$ be a matroid on $E$ of rank $d$ for each $\alpha\in \Z^E$. The following are equivalent:
\begin{enumerate}
\item $M:\alpha\mapsto M_\alpha$ is a matroid flock; and
\item there is a matroid valuation $\nu:\binom{E}{d} \rightarrow\Zinf$ so that $M_\alpha=M^\nu_\alpha$ for all $\alpha\in \Z^E$.
\end{enumerate}
\end{theorem}
In what follows, we will first prove this characterization, and then derive a few facts about matroid flocks which will be used later in the paper.

\subsection{The characterization of matroid flocks} The implication (2)$\Rightarrow$(1) of Theorem \ref{thm:mf_char} is relatively straightforward.
\begin{lemma} Let $\nu:\binom{E}{d} \rightarrow\Zinf$ be a valuation. Then
\begin{enumerate}
\item $M^\nu_\alpha/ i=M^\nu_{\alpha+e_i}\del i$ for all $\alpha\in\Z^E$ and $i\in E$; and
\item $M^\nu_\alpha=M^\nu_{\alpha+\one}$ for all $\alpha\in \Z^E$.
\end{enumerate}
\end{lemma}
\proof We prove (1). Consider $\alpha\in\Z^E$ and $i\in E$. If $i$ is a loop of $M_\alpha^\nu$, then  $g^\nu(\alpha+e_i)=g^\nu(\alpha)$. Then $B$ is a basis of $M^\nu_\alpha/ i$ if and only if
$$e^T_B(\alpha+e_i)-\nu(B)=e^T_B\alpha-\nu(B)=g^\nu(\alpha)=g^\nu(\alpha+e_i)$$
if and only if
$B$ is a basis of $M^\nu_{\alpha+e_i}\del i$.
On the other hand, if $i$ is not a loop of $M_\alpha^\nu$, then  $g^\nu(\alpha+e_i)=g^\nu(\alpha)+1$, and then $B'$ is a basis of $M^\nu_\alpha/ i$ if and only if $B=B'+i$ is a basis of $M_\alpha^\nu$, if and only if
$$e^T_B(\alpha+e_i)-\nu(B)=e^T_B\alpha-\nu(B)+1=g(\alpha)+1=g(\alpha+e_i)$$ if and only if $B'$ is a basis of $M^\nu_{\alpha+e_i}\del i$.

To see (2), note that $g^\nu(\alpha+\one)=g^\nu(\alpha)+d$. Then $B$ is a basis of $M^\nu_\alpha$ if and only if $$e^T_B(\alpha+\one)-\nu(B)=e^T_B\alpha-\nu(B)+d=g^\nu(\alpha)+d=g^\nu(\alpha+\one)$$
if and only if $B$ is a basis of $M^\nu_{\alpha+\one}$.
\endproof

We show the implication (1)$\Rightarrow$(2) of Theorem \ref{thm:mf_char}. Our proof makes essential use of the discrete duality theory of Murota \cite{MurotaBook}. Specifically, we will first construct an L-convex function $g: \Z^E\rightarrow \Z$ from the matroid flock $M$. The Fenchel dual $f$ of $g$ is then an M-convex function, from which we derive the required valuation $\nu$.
 Before we can prove the existence of a suitable function $g$, we need a few technical lemmas.

In the context of a matroid flock $M$, we will write $r_\alpha$ for the rank function of the matroid $M_\alpha$.
We first extend (MF1).
\begin{lemma}\label{lem:mf}Let $M$ be a matroid flock on $E$, let $\alpha\in\Z^E$ and let $I\subseteq E$. Then $M_\alpha/ I=M_{\alpha+e_I}\del I$.
\end{lemma}
\proof By induction on $I$. Clearly the lemma holds if $I=\emptyset$. If $I\neq\emptyset$, pick any $i\in I$. Using the induction hypothesis for $\alpha$ and $I-i$, followed by (MF1) for $\alpha+e_{I-i} $ and $i$, we have
$$M_\alpha/I=M_\alpha/(I-i)/i=M_{\alpha+e_{I-i}}\del (I-i)/i=M_{\alpha+e_{I-i}}/i\del (I-i)=M_{\alpha+e_I}\del i\del (I-i)=M_{\alpha+e_I}\del I,$$
as required.\endproof
\begin{lemma} \label{lem:triangle}Let $M$ be a matroid flock on $E$, let $\alpha\in\Z^E$ and let $I\subseteq J\subseteq E$. Then
$$r_{\alpha}(J)=r_{\alpha}(I)+r_{\alpha+e_I}(J\del I).$$
\end{lemma}
\proof Using Lemma \ref{lem:mf}, we have
$r_{\alpha+e_I}(J\del I)=r(M_{\alpha+e_I}\del I\del\overline{J})=r(M_\alpha/I\del\overline{J})=r_{\alpha}(J)-r_{\alpha}(I),$ as required.
\endproof
If $\alpha,\beta\in \R^E$, we write $\alpha\leq \beta$ if $\alpha_i\leq \beta_i$ for all $i\in E$.
\begin{lemma}\label{lem:rank}Let $M$ be a matroid flock on $E$, let $\alpha, \beta\in\Z^E$ and let $I\subseteq E$. If $I\cap \supp(\beta-\alpha)=\emptyset$ and $\alpha\leq \beta$, then $r_\alpha(I)\geq r_\beta(I)$.
\end{lemma}
\proof We use induction on $\max_i (\beta_i-\alpha_i)$. Let $J:=\supp(\beta-\alpha)$. Then
$$r_\alpha(I)=r(M_\alpha\del J\del\overline{I\cup J})\geq r(M_\alpha/J\del\overline{I\cup J})=r(M_{\alpha+e_J}\del J\del\overline{I\cup J})=r_{\alpha+e_J}(I).$$
Taking $\alpha':=\alpha+e_J$, we have $\alpha'\leq \beta$, $\max_i (\beta_i-\alpha'_i)<\max_i (\beta_i-\alpha_i)$ and $\supp(\beta-\alpha')\subseteq \supp(\beta-\alpha)$. Hence
$$r_\alpha(I)\geq r_{\alpha+e_J}(I)\geq r_\beta(I),$$
by using the induction hypothesis for $\alpha', \beta$.\endproof
We are now ready to show the existence of the function $g$ alluded to above.
\begin{lemma} \label{lem:mf_g}Let $M$ be a matroid flock on $E$. There is a unique function $g:\Z^E\rightarrow\Z$ so that
\begin{enumerate}
\item $g(0)=0$, and
\item $g(\alpha+e_I)=g(\alpha)+r_{\alpha}(I)$ for all $\alpha\in\Z^E$ and $I\subseteq E$.
\end{enumerate}
\end{lemma}
\proof
Let $D=\left(\Z^E,A\right)$ be the infinite directed graph with  arcs
$$A:=\{(\alpha, \alpha+e_I): ~\alpha\in \Z^E, ~\emptyset\neq I\subseteq E\}.$$
Let $l:A\rightarrow \Z$ be a length function on the arcs determined by $l(\alpha, \alpha+e_I)=r_{\alpha}(I)$.
This length function extends to the undirected walks $W=(\alpha^0,\ldots,\alpha^k)$ of $D$ in the usual way, by setting
$$l(W):=\left\{\begin{array}{ll} l(\alpha^0,\ldots,\alpha^{k-1})+l(\alpha^{k-1}, \alpha^k)&\text{if } (\alpha^{k-1},\alpha^{k})\in A\\
l(\alpha^0,\ldots,\alpha^{k-1})-l(\alpha^{k}, \alpha^{k-1})&\text{if } (\alpha^k, \alpha^{k-1})\in A
\end{array}\right.$$
if $k>0$, and $l(W)=0$ otherwise. A walk $(\alpha^0,\ldots, \alpha^k)$ is {\em closed} if it starts and ends in the same vertex, i.e. if $\alpha^0=\alpha^k$.

If we assume that $l(W)=0$ for each closed walk $W$, then we can construct a function $g$ satisfying (1) and (2) as follows. For each $\alpha\in \Z^E$, let $W^\alpha$ be an arbitrary walk from $0$ to $\alpha$, and put $g(\alpha)=l(W^\alpha)$. Then $g(0)=l(W^0)=0$ by our assumption, since $W^0$ is a walk from $0$ to $0$. Also, if $\alpha\in \Z^E$ and $I\subseteq E$, then writing $\beta:=\alpha+e_I$ we have
$$l(W^\alpha)+l(\alpha, \beta)-l(W^\beta)=l(\alpha^0,\ldots, \alpha^k, \beta^m,\ldots, \beta^0)=0,$$
by our assumption, where $W^\alpha=(\alpha^0,\ldots, \alpha^k)$ and $W^\beta=(\beta^0,\ldots, \beta^m)$. It follows that $$g(\alpha+e_I)=l(W^\beta)=l(W^\alpha)+l(\alpha, \beta)=g(\alpha)+r_{\alpha}(I),$$
as required. So to prove the lemma, it will suffice to show that $l(W)=0$ for each closed walk $W$.

Suppose for a contradiction that $W=(\alpha^0,\ldots, \alpha^k)$ is a closed walk with $l(W)\neq 0$. Fix any $i\in E$ so that $\alpha^0_i\neq \alpha^1_i$. If $J\subseteq E$ is such that $i\in J$, then for any $\alpha\in \Z^E$ we have
\begin{equation}\label{eq:triangle}l(\alpha, \alpha+e_J)=r_{\alpha}(J)=r_{\alpha}(i)+r_{\alpha+e_i}(J-i)=l(\alpha, \alpha+e_i, \alpha+e_J)\end{equation}
by applying Lemma \ref{lem:triangle} with $I=\{i\}$. Hence, if we replace each subsequence
$(\alpha^{t-1},\alpha^t)=(\alpha, \alpha+e_J)$ of $W$ with $i\in J$ by  $(\alpha, \alpha+e_i, \alpha+e_J)$, and each subsequence $(\alpha^{t-1},\alpha^t)=(\alpha+e_J, \alpha)$ with $i\in J$ by $(\alpha+e_J, \alpha+e_i, \alpha)$, then we obtain a closed walk $U=(\beta^0,\ldots, \beta^m)$ with $l(U)=l(W)\neq 0$,  such that if $\beta^t-\beta^{t-1}=\pm e_J$, then $i\not\in J$ or $J=\{i\}$, and moreover such that $\beta^t-\beta^{t-1}=\pm e_i$ for some $t$. Pick such $U,i$ with $m$ as small as possible, and minimizing $|U|_i:=\sum\{t\in \{1,\ldots, m\}:  \beta^t_i\neq \beta^{t-1}_i\}.$

We claim that there is no $t>0$ so that $\beta^{t-1}_i=\beta^{t}_i\neq \beta^{t+1}_i$. Consider that by applying Lemma \ref{lem:triangle} with $I=J-i$, we have
\begin{equation}\label{eq:triangle2}l(\alpha, \alpha+e_J)=r_{\alpha}(J)=r_{\alpha}(J-i)+r_{\alpha+e_{J-i}}(i)=l(\alpha, \alpha+e_{J-i}, \alpha+e_J), \end{equation}
so that using \eqref{eq:triangle} we obtain $l(\alpha, \alpha+e_{J-i}, \alpha+e_J,\alpha+e_i,\alpha)=0$. Hence, $l(\alpha, \alpha+e_{J-i}, \alpha+e_J)=l(\alpha, \alpha+e_i, \alpha+e_J)$ and
 $l(\alpha+e_{J-i}, \alpha, \alpha+e_i)=l(\alpha+e_{J-i}, \alpha+e_J, \alpha+e_i)$. It follows that
 any subsequence $(\beta^{t-1}, \beta^{t}, \beta^{t+1})$ of $U$ with $\beta^{t-1}_i=\beta^{t}_i\neq \beta^{t+1}_i$ can be rerouted to $(\beta^{t-1}, \beta', \beta^{t+1})$ with $\beta^{t-1}_i\neq \beta'_i=\beta^{t+1}_i$, which would result in a closed walk $U'$ with $|U'|_i<|U|_i$, a contradiction.

 So there exists an $m'\in \{1,\ldots, m\}$ such that $\beta^t-\beta^{t-1}=\pm e_i$ if and only if $t\leq m'$. Then $\beta^0=\beta^{m'}=\beta^m$, and $l(\beta^0,\ldots, \beta^{m'})=0$. Hence
 $$l(\beta^{m'},\ldots, \beta^{m})= l(\beta^0,\ldots, \beta^{m'})+l(\beta^{m'},\ldots, \beta^{m})=l(U)\neq 0,$$ which contradicts the minimality of $m$.
\endproof

For any matroid flock $M$, let $g^M$ denote the unique function $g$  from Lemma \ref{lem:mf_g}.
\begin{theorem} \label{thm:Lconvex}Let $M$ be a matroid flock of rank $d$ on $E$, and let $g=g^M$. Then
\begin{enumerate}
  \item  $g(\alpha) + g(\beta) \geq g(\alpha \vee \beta) + g(\alpha \wedge \beta)$ for all $\alpha,\beta \in \Z^E$; and
  \item $g(\alpha + \one) = g(\alpha) + d$ for all $\alpha\in \Z^E$.
\end{enumerate}
\end{theorem}
\proof We first show (1).  Let $\alpha, \beta\in \Z^E$. Since $(\beta - \alpha) \vee 0\geq 0$, there are $I_1 \subseteq \ldots \subseteq I_k \subseteq E$ so that $(\beta - \alpha) \vee 0 = \sum_{j=1}^k e_{I_j}$. Let $\gamma(t) := \alpha\wedge \beta+ \sum_{j=1}^t e_{I_j}$. Then
$\gamma(0)=\alpha\wedge \beta$, $\gamma(k)=\alpha\wedge \beta+ (\beta - \alpha) \vee 0= \beta$, and
$\gamma(t)=\gamma(t-1)+e_{I_t}$, so that
$$g(\beta)-g(\alpha\wedge\beta)=\sum_{t=1}^k g( \gamma(t)) - g(\gamma(t-1))= \sum_{t=1}^{k} r_{\gamma(t-1)}(I_t).$$
Let $\delta:= (\alpha-\beta)\vee 0$. Then $\gamma(0)+\delta=\alpha$ and $\gamma(k)+\delta=\alpha\vee \beta$, and we also have
$$g(\alpha\vee \beta)-g(\alpha)=\sum_{t=1}^k g( \gamma(t)+\delta) - g(\gamma(t-1)+\delta)= \sum_{t=1}^k r_{\gamma(t-1)+\delta}(I_t).$$
For each $t$ we have $I_t\cap \supp(\delta)\subseteq \supp((\beta-\alpha)\vee 0)\cap \supp((\alpha-\beta)\vee 0)=\emptyset$, and $\delta\geq 0$. By Lemma \ref{lem:rank}, it follows that $r_{\gamma(t-1)}(I_t)\geq r_{\gamma(t-1)+\delta}(I_t)$ for each $t$, and hence
$$g(\beta)-g(\alpha\wedge\beta)= \sum_{t=1}^k r_{\gamma(t)}(I_t)\geq \sum_{t=1}^k r_{\gamma(t)+\delta}(I_t)=g(\alpha\vee \beta)-g(\alpha),$$
which implies (1).

 To see (2), note that $g(\alpha+\one)=g(\alpha)+r_{\alpha}(E)=g(\alpha)+d$.
\endproof
It follows that for any matroid flock $M$, the function $g^M$ is L-convex in the sense of Murota.
\begin{lemma}\label{lem:dual}Let $M$ be a matroid flock on $E$, let $g=g^M$ and $f:=g^\bullet$, and let $\alpha, \omega\in \Z^E$. The following are equivalent.
\begin{enumerate}
\item $\omega^T\alpha=f(\omega)+g(\alpha)$; and
\item $\omega=e_B$ for some basis $B$ of $M_\alpha$.
\end{enumerate}
\end{lemma}
\proof We first show that (1) implies (2). So assume that $\omega^T\alpha=f(\omega)+g(\alpha)$.
Then $f(\omega)$ is finite, as $g(\alpha)$ and $\omega^T\alpha$ are both finite. Since $f=g^\bullet$, we have
$$\omega^T\alpha-g(\alpha)=f(\omega)=\sup\{\omega^T\beta-g(\beta): \beta\in \Z^E\},$$
and hence $\alpha$ minimizes the function $G: \beta\mapsto g(\beta)-\omega^T\beta$ over all $\beta\in\Z^E$.
Since
$$0\leq G(\alpha-e_i)-G(\alpha)=g(\alpha-e_i)-\omega^T(\alpha-e_i)-g(\alpha)+\omega^T\alpha=-r_{\alpha-e_i}(i)+ \omega_i$$
for each $i\in E$, it follows that $\omega\geq 0$.
Since
$$0\leq G(\alpha+e_i)-G(\alpha)=g(\alpha+e_i)-\omega^T(\alpha+e_i)-g(\alpha)+\omega^T\alpha=r_\alpha(i)-\omega_i,$$
we have $\omega\leq \one$. Hence $\omega=e_B$ for some $B\subseteq E$.
Then
$$0\leq G(\alpha+e_B)-G(\alpha)=g(\alpha+e_B)-\omega^T(\alpha+e_B)-g(\alpha)+\omega^T\alpha=r_{\alpha}(B)-|B|,$$
so that $r_{\alpha}(B)=|B|$.
Moreover,
$$0\leq G(\alpha-e_{\overline{B}})-G(\alpha)=g(\alpha-e_{\overline{B}})-\omega^T(\alpha-e_{\overline{B}})-g(\alpha)+\omega^T\alpha=-r_{\alpha-e_{\overline{B}}}(\overline{B}),$$
so that $r_{\alpha-e_{\overline{B}}}(\overline{B})=0$.
It follows by Lemma \ref{lem:mf} that
$$|B|=r_\alpha(B)=r(M_\alpha\del{\overline{B}})=r(M_{\alpha-e_{\overline{B}}}/\overline{B})=d-r_{\alpha-e_{\overline{B}}}(\overline{B})=d,$$
and hence  that $B$ is a basis of $M_\alpha$.

We now show that (2) implies (1). Suppose $\omega=e_B$ for some basis $B$ of $M_\alpha$. Consider again the function $G: \alpha\mapsto g(\alpha)-\omega^T\alpha$ over $\Z^E$. As $g$ is L-convex, $G$ is L-convex. We show that $\alpha$ minimizes $G$ over $\Z^E$, using the optimality condition for L-convex functions given in Lemma \ref{thm:Llocalopt}. First, note that as $g(\alpha+\one)=g(\alpha)+d$, we have
$$G(\alpha+\one)=g(\alpha+\one)-\omega^T(\alpha+\one) =g(\alpha)+d - \omega^T\alpha-|B|= G(\alpha).$$
Let $I\subseteq E$. As $B$ is a basis of $M_\alpha$, we have $|B\cap I|\leq r_\alpha(I)$, and hence
$$G(\alpha+e_I)-G(\alpha)=g(\alpha+e_I)-\omega^T(\alpha+e_I)-g(\alpha)-\omega^T\alpha=r_\alpha(I) -|B\cap I|\geq 0.$$
 Thus $\alpha$ minimizes $G$ over $\Z^E$, hence $f(\omega)=\sup\{-G(\alpha):\alpha\in \Z^E\}=\omega^T\alpha-g(\alpha)$, as required.
 \endproof

Let $M$ be a matroid flock on $E$ of rank $d$. We define the function $\nu^M:\binom{E}{d}\rightarrow \Zinf$ by setting $\nu^M(B):=f(e_B)$  for each $B\in \binom{E}{r}$, where $f=g^\bullet$ is the Lagrange-Fenchel dual of $g=g^M$.
\begin{lemma}\label{lem:nuM}Let $M$ be a matroid flock, and let $\nu=\nu^M$. Then $\nu$ is a valuation,   and $M^\nu_\alpha=M_\alpha$ for all $\alpha\in\Z^E$.\end{lemma}
\proof Suppose $M$ is a matroid flock. Then $g=g^M$ is L-convex by Theorem \ref{thm:Lconvex}, and $f:=g^\bullet$ is M-convex by Theorem \ref{thm:MLdual}. That $\nu:B\mapsto f(e_B)$ is a matroid valuation is straightforward from the fact that $f$ is M-convex.   We show that $M^\nu_\alpha=M_\alpha$ for all $\alpha\in\Z^E$.  By Theorem \ref{thm:MLdual}, we have $g=f^\bullet$.  By Lemma \ref{lem:dual}, we have
$$g(\alpha)=f^\bullet(\alpha)=\sup\{\omega^T\alpha-f(\omega): \omega\in \Z^E\}=
\sup\left\{e_{B'}^T\alpha-\nu(B'): B'\in\binom{E}{r}\right\},$$
as the first supremum is attained by $\omega$ only if $\omega=e_{B'}$ for some $B'\in\binom{E}{r}$.
Again by Lemma \ref{lem:dual}, $B$ is a basis of $M_\alpha$ if and only if $g(\alpha)=e_B^T\alpha-\nu(B)$, i.e. if $B$ is a basis of $M^\nu_\alpha$.
\endproof
This proves the implication (1)$\Rightarrow$(2) of Theorem \ref{thm:mf_char}. Finally, we note:
\begin{lemma}Let $M$ be a matroid flock, and let $\nu=\nu^M$. Then $g^M=g^\nu$.\end{lemma}
\proof Let $f:\Z^E\rightarrow\Zinf$ be defined by $f(e_B)=\nu(B)$ for all $B\in \binom{E}{d}$, and $=\infty$ otherwise. Then $g^M=f^\bullet=g^\nu$, as required.\endproof

\subsection{The support matroid and the cells of a matroid
valuation} \label{sec:Cells}

If $M$ is a matroid flock, then the {\em support matroid of $M$} is just the support matroid $M^\nu$ of the associated valuation $\nu=\nu^M$.
\begin{lemma} \label{lem:supp} Suppose $M:\alpha\mapsto M_\alpha=(E,\BB_\alpha)$ is a matroid flock with support matroid $N=(E, \BB)$. Then $\BB=\bigcup_{\alpha\in \Z^E}\BB_\alpha$.
\end{lemma}
\proof Let $\nu=\nu^M$, and put $g=g^\nu$. By Lemma \ref{lem:nuM}, we have
$$\BB_\alpha=\BB^\nu_\alpha=\left\{B\in \binom{E}{d}: e_B^T\alpha-\nu(B)=g(\alpha)\right\}$$
for all $\alpha\in\Z^E$, and $\BB=\{B\in\binom{E}{d}: \nu(B)<\infty\}$.
Since $\nu(B')<\infty$ for some $B'\in \binom{E}{d}$ by (V1), we have  $g(\alpha)>-\infty$ for all $\alpha\in \Z^E$. Consider a $B\in \binom{E}{d}$.

Suppose first that  $B\in \BB$, i.e. $\nu(B)<\infty$. Consider  the difference $h(\alpha):=g(\alpha)-e^T_B\alpha+\nu(B)$. Then $h(\alpha)$ is nonnegative and finite for all $\alpha\in \Z^E$, and $B\in \BB_\alpha$ if and only if $h(\alpha)=0$. Moreover, if $B$ is not a basis of $M^\nu_\alpha$, then $g(\alpha+e_B)\leq g(\alpha)+|B|-1$ and $e^T_B(\alpha+e_B)=e^T_B\alpha+|B|$, so that
$h(\alpha+e_B)\leq h(\alpha)-1$. It follows that for any fixed $\alpha$ and any sufficiently large $k\in \Z$, we have  $h(\alpha+ke_B)=0$, and then $B\in \BB_{\alpha+ke_B}$. Then $B\in \bigcup_{\alpha\in \Z^E}\BB_\alpha$.

If on the other hand $B\not\in \BB^\nu$, i.e. $\nu(B)=\infty$, then
$e^T_B\alpha-\nu(B)=-\infty<g(\alpha)$ for all $\alpha\in \Z^E$, so that $B\not\in \BB_\alpha$ for any $\alpha\in \Z^E$. Then $B\not\in\bigcup_{\alpha\in \Z^E}\BB_\alpha$.
\endproof
The geometry of valuations is quite intricate, and is studied in much greater detail in tropical geometry \cite{Speyer2008, Hampe2015}. We mention only a few results we need in this paper.
 For any matroid valuation $\nu:\binom{E}{d}\rightarrow \Rinf$, put
$C^\nu_\beta:=\{\alpha\in \R^E: \BB^\nu_\alpha\supseteq \BB^\nu_\beta\}.$
\begin{lemma}\label{lem:poly} Let $\nu:\binom{E}{d}\rightarrow \Zinf$ be a matroid valuation,  and let $\beta\in\R^E$. Then $$C^\nu_\beta=\{\alpha\in \R^E: \alpha_i-\alpha_j\geq \nu(B)-\nu(B')\text{ for all } B\in \BB^\nu_\beta, ~B'\in \BB^\nu\text{ s.t. }B'=B-i+j\}.$$
\end{lemma}
\proof Let $C$ denote the right-hand side polyhedron in the statement of the lemma. Directly from the definition of $\BB^\nu_\alpha$, it follows that $\BB^\nu_\alpha\supseteq \BB^\nu_\beta$ if and only if $$e^T_B\alpha-\nu(B)\geq e^T_{B'}\alpha-\nu(B')$$ for all $B\in\BB^\nu_\beta$ and $B'\in \BB^\nu$.
In particular, $C^\nu_\beta\subseteq C$.

To see that  $C^\nu_\beta\supseteq C$, suppose that  $\alpha\not\in C^\nu_\beta$, that is,  $\BB^\nu_\alpha\not\supseteq \BB^\nu_\beta$, so that
$$e^T_B\alpha-\nu(B)< e^T_{B'}\alpha-\nu(B')$$
for some $B\in\BB^\nu_\beta$ and $B'\in \BB^\nu$.
 Consider the valuation $\nu':B\mapsto \nu(B)-e^T_B\alpha$. Pick $B\in\BB^\nu_\beta, B'\in \BB^\nu$ such that $\nu'(B)>\nu'(B')$ with $B\del B'$ as small as possible.  If $|B\del B'|>1$,  then  by minimality of $|B\del B'|$ we have
$$\nu'(B)+\nu'(B)>\nu'(B)+\nu'(B')\geq \nu'(B-i+j)+\nu'(B'+i-j)\geq \nu'(B)+\nu'(B),$$
for some $i\in B\del B'$ and $j\in B'\del B$, since $\nu'$ is a valuation. This is a contradiction, so $|B\del B'|=1$ and $B'=B-i+j$, and hence
$$\alpha_i-\alpha_j= (e^T_B-e^T_{B'})\alpha<\nu(B)-\nu(B'),$$
so that $\alpha\not\in C$.
\endproof
Thus the cells $C^\nu_\beta$ are `alcoved polytopes' (see \cite{LamPostnikov2007}). The relative interior of such cells is connected also in a discrete sense.
\begin{lemma} \label{lem:walk}Let $M:\alpha\mapsto M_\alpha$ be a matroid flock on $E$, and let $\alpha,\beta\in \Z^E$. If $M_\alpha=M_\beta$, then there is a walk $\gamma^0,\ldots, \gamma^k\in\Z^E$ from $\alpha=\gamma^0$ to $\beta=\gamma^k$ so that $M_{\gamma^i}=M_\alpha$ for $i=0,\ldots, k$, and for each $i$ there is a $J_i$ so that $\gamma^i-\gamma^{i-1}=\pm e_{J_i}$
\end{lemma}
\proof By (MF2), there is a feasible walk from $\alpha$ to $\alpha+k\one$ for any $k\in\Z$, taking steps of the form $\pm\one$. Fixing any $i_0\in E$, we may assume that $\alpha_{i_0}=0$, and similarly that $\beta_{i_0}=0$.

Let $\nu=\nu^M$. Using Lemma \ref{lem:poly}, we have $\{\gamma\in \R^E: M_\gamma=M_\alpha\}=(C^\nu_\alpha)^\circ,$ where $(C^\nu_\alpha)^\circ$ denotes the relative interior of $C^\nu_\alpha$.
For each $i,j$ let $c_{ij}:=\min\{\alpha_i-\alpha_j, \beta_i-\beta_j\}$. Then by inspection of the system of inequalities which defines $C^\nu_\alpha$ (Lemma \ref{lem:poly}), we have
$$\alpha, \beta\in C:=\{\gamma\in \R^E: \gamma_i-\gamma_j\geq c_{ij}\text{ for all }i,j, \text{ and }\gamma_{i_0}=0\}\subseteq \{\gamma\in \R^E: M_\gamma=M_\alpha\}.$$
Then  $C$ is a bounded polyhedron defined by a totally unimodular system of inequalities with integer constant terms $c_{ij}$. It follows that $C$ is an integral polytope. Moreover $\alpha, \beta$ are both vertices of $C$, and hence there is a walk from $\alpha$ to $\beta$ over the 1-skeleton of $C$. Since $C$ has integer vertices, and each edge of $C$ is parallel to $e_J$ for some $J\subseteq E$, the lemma follows.
\endproof
\begin{lemma} \label{lem:step}Let $M:\alpha\mapsto M_\alpha$ be a matroid flock on $E$, let $\alpha\in \Z^E$, and let $J\subseteq E$. If $M_\alpha=M_{\alpha+e_J}$, then $\lambda_{M_\alpha}(J)=0$.\end{lemma}
\proof By Lemma \ref{lem:mf}, we have $M_\alpha/J=M_{\alpha+e_J}\del J=M_\alpha\del J$. The lemma follows.\endproof

\section{Frobenius flocks from algebraic matroids} \label{sec:algebraic}

Let $E$ be a finite set, and let $K$ be an algebraically closed field of positive characteristic $p>0$, so that the Frobenius map $F:x\mapsto x^p$ is an automorphism of $K$. For any $w\in K^E$ and $\alpha\in \Z^E$, let $$\alpha w:=(F^{-\alpha_i}(w_i))_{i\in E},$$ and for a subset $W\subseteq K^E$,  let $\alpha W:=\{ \alpha w: w\in W\}.$

A {\em Frobenius flock of rank $d$ on $E$ over $K$} is a map $V$ which assigns a linear subspace $V_\alpha \subseteq K^E$ of dimension $d$ to each $\alpha\in\Z^E$, such that
\begin{itemize}
\item[(FF1)] $V_{\alpha}/ i=V_{\alpha+e_i}\del i$ for all $\alpha\in\Z^E$ and $i\in E$; and
\item[(FF2)] $V_{\alpha+\one}=\one V_{\alpha} $ for all $\alpha\in \Z^E$.
\end{itemize}
By inspection of the definition of a matroid flock, it is clear that the map $\alpha \mapsto M(V_\alpha)$ is a matroid flock. We write $M(V)$ for the support matroid of this associated matroid flock.
In view of Lemma \ref{lem:supp}, we may consider a flock $V$ as a representation of its support matroid $M(V)$.

Our interest in Frobenius flocks originates from
the observation that they can represent arbitrary algebraic matroids.
This is especially interesting since in positive
characteristic no other
``linear'' representation of arbitrary algebraic matroids was available so
far. In this section, we first introduce some preliminaries on algebraic
matroids, and then establish their representability by Frobenius flocks.

\subsection{Preliminaries on algebraic matroids}

Let $K$ be a field and $L$ an extension field of $K$.  Elements
$a_1,\ldots,a_n \in L$ are called {\em algebraically independent} over
$K$ if there exists no nonzero polynomial $f \in K[x_1,\ldots,x_n]$
such that $f(a_1,\ldots,a_n)=0$. Algebraic independence satisfies the
matroid independence axioms (I1)---(I3) \cite{OxleyBook}, and this leads to the following notion.

\begin{definition}
Let $M$ be a matroid on a finite set $E$. An {\em algebraic
representation} of $M$ over $K$ is a pair $(L,\phi)$ consisting of a
field extension $L$ of $K$ and a map $\phi:E \to L$ such that any $I
\subseteq E$ is independent in $M$ if and only if the multiset $\phi(I)$
is algebraically independent over $K$.
\end{definition}

For our purposes it will be useful to take a more geometric viewpoint;
a good general reference for the algebro-geometric terminology that
we will use is \cite{Cox2007}; and we refer to \cite{Kiralyi2013,Rosen2014}
for details on the link to algebraic matroids.

First, we assume throughout this section that $K$ is algebraically
closed. This is no loss of generality in the following sense: take an
algebraic closure $L'$ of $L$ and let $K'$ be the algebraic closure of $K$
in $L'$. Then for any subset $I \subseteq E$ the set $\phi(I) \subseteq
L$ is algebraically independent over $K$ if and only if $\phi(I)$ is
algebraically independent over $K'$.

Second, there is clearly no harm in assuming that $L$ is generated by
$\phi(E)$. Then let $P$ be the kernel of the $K$-algebra homomorphism
from the polynomial ring $R:=K[(x_i)_{i \in E}]$ into $L$ that sends
$x_i$ to $\phi(i)$. Since $L$ is a domain, $P$ is a prime ideal, so
the quotient $R/P$ is a domain and $L$ is isomorphic to the field of
fractions of this domain.

By Hilbert's basis theorem, $P$ is finitely generated, and one can store
algebraic representations on a computer by means of a list of generators
(of course, this requires that one can already compute with elements of
$K$). In these terms, a subset $I \subseteq E$ is independent if and
only if $P \cap K[x_i : i \in I]=\{0\}$. Given generators of $P$,
this intersection can be computed using Gr\"obner bases
\cite[Chapter 3, \S 1, Theorem 2]{Cox2007}.

The vector space $K^E$ is equipped with the Zariski topology, in which
the closed sets are those defined by polynomial equations; we will
use the term {\em variety} or {\em closed subvariety} for such a set.
In particular, let $X \subseteq K^E$ be the closed subvariety defined as
the $\{a \in K^E : \forall f \in P: f(a)=0\}$. Since $P$ is prime, $X$
is an irreducible closed subvariety, and by Hilbert's Nullstellensatz,
$P$ is exactly the set of all polynomials that vanish everywhere on $X$.

We have now seen how to go from an algebraic representation over $K$
of a matroid on $E$ to an irreducible subvariety of $K^E$.  Conversely,
every irreducible closed subvariety $Y$ of $K^E$ determines an algebraic
representation of some matroid, as follows: let $Q \subseteq R$ be
the prime ideal of polynomials vanishing on $Y$, let $K[Y]:=R/Q$ be
the integral domain of regular functions on $Y$, and set $L:=K(Y)$, the fraction
field of $K[Y]$. The map $\phi$ sending $i$ to the class of $x_i$ in
$L$ is a representation of the matroid $M$ in which $I \subseteq E$
is independent if and only if $Q \cap K[x_i : i \in I]=\{0\}$. This
latter condition can be reformulated as saying that the image of $Y$
under the projection $\pi_I:K^E \to K^I$ is dense in the latter space,
i.e., that $Y$ projects {\em dominantly} into $K^I$. Our discussion is
summarised in the following definition and lemma.

\begin{definition}
An {\em algebro-geometric representation} of a matroid $M$ on the ground
set $E$ over the algebraically closed field $K$ is an irreducible, closed
subvariety $Y$ of $K^E$ such that $I \subseteq E$ is independent in $M$
if and only if $\overline{\pi_I(Y)}=K^I$. We denote the matroid $M$
represented by $Y$ as $M(Y)$.
\end{definition}

We have seen:

\begin{lemma}
A matroid $M$ admits an algebraic representation over the algebraically
closed field $K$ if and only if it admits an algebro-geometric
representation over $K$. \hfill $\square$
\end{lemma}

The rank function on $M$ corresponds to dimension:

\begin{lemma}
If $Y$ is an algebro-geometric representation of $M$, then for each
$I \subseteq E$ the rank of $I$ in $M$ is the dimension of the Zariski
closure $\overline{\pi_I(Y)}$.  \hfill $\square$
\end{lemma}

Because of this equivalence between algebraic and algebro-geometric
representations, we will continue to use the term algebraic representation
for algebro-geometric representations.

\subsection{Tangent spaces}

Crucial to our construction of a flock from an algebraic matroid are
tangent spaces, which were also used in \cite{Lindstrom1985} in the
study of characteristic sets of algebraic matroids.  In this subsection,
$Y \subseteq K^E$ is an irreducible, closed subvariety with vanishing
ideal $Q \subseteq R$, $K[Y]:=R/Y$ is its coordinate ring, and $K(Y)$
its function field.

\begin{definition}
Define the $K[Y]$-module
\[ J_Y:=\left\{\left(\frac{\partial f}{\partial x_j} + Q\right)_{j \in E} : f \in Q\right\}
\subseteq K[Y]^E. \]
For any $v \in Y$, define the tangent space $T_v Y := J_Y(v)^\perp
\subseteq K^E$, where $J_Y(v) \subseteq K^E$ is the image of $J_Y$ under
evaluation at $v$.
Let $\eta \in K(Y)^E$ be the generic point of $Y$,
i.e., the point $(x_j+Q)_{j \in E}$, and define $T_\eta Y$ as $(K(Y)
\otimes_{K[Y]} J_Y)^\perp \subseteq K(Y)^E$. The variety $Y$ is called
{\em smooth} at $v$ (and $v$ a {\em smooth point} of $Y$) if $\dim_K
T_v Y = \dim_{K(Y)} T_\eta Y$.
\end{definition}

The $K[Y]$-module $J_Y$ is generated by the rows of the Jacobi matrix
$(\frac{\partial f_i}{\partial x_j} + Q)_{i,j}$ for any finite set of
generators $f_1,\ldots,f_r$ of $Q$. The right-hand side in the smoothness
condition also equals the transcendence degree of $K(Y)$ over $K$ and
the Krull dimension of $Y$. The smooth points in $Y$ form an open and
dense subset of $Y$.

We recall the following property of smooth points.

\begin{lemma} \label{lm:Saturated}
Assume that $Y$ is smooth at $v \in Y$ and let $S$ be the local ring
of $Y$ at $v$, i.e., the subring of $K(Y)$ consisting of all fractions
$f/g$ where $g(v) \neq 0$. Then $M:=S \otimes_{K[Y]} J_Y \subseteq S^E$
is a free $S$-module of rank equal to $|E|-\dim Y$, which is
saturated in the sense that $su \in M$ for $s \in S$ and $u
\in S^E$ implies $u \in M$. \hfill $\square$
\end{lemma}

Now we come to a fundamental difference between characteristic zero and
positive characteristic.

\begin{lemma}
Let $v \in Y$ be smooth, and let $I \subseteq E$. Then $\dim \pi_I(T_v Y)
\leq \dim \overline{\pi_I(Y)}$. If, moreover, the characteristic of $K$
is equal to zero, then the set of smooth points $v \in Y$ for which
equality holds is an open and dense subset of $Y$. \hfill $\square$
\end{lemma}

The inequality is fairly straightforward, and it shows that
$M(T_v Y)$ is always a weak image of $M(Y)$ (of the same
rank as the latter).  For a proof of the statement
in characteristic zero, see for instance \cite[Chapter II, Section
6]{Shafarevich1994}. A direct consequence of this lemma is the
following, well-known theorem \cite{Ingleton1971}.

\begin{theorem}[Ingleton] \label{thm:CharZero}
If $\cha K=0$, then every matroid that admits an algebro-geometric
representation over $K$ also admits a linear representation over $K$.
\end{theorem}

\begin{proof}
If $Y \subseteq K^E$ represents $M$, then the set of smooth points $v
\in Y$ such that $\dim \pi_I(T_v Y)=\dim \overline{\pi_I(M)}$ for all
$I \subseteq E$ is a finite intersection of open, dense subsets, and
hence open and dense. For any such point $v$, the linear space $T_v Y$
represents the same matroid as $Y$.
\end{proof}

A fundamental example where this reasoning fails in positive
characteristic was given in the introduction.
As we will see next, Frobenius flocks are a variant
of Theorem~\ref{thm:CharZero} in positive characteristic.

\subsection{Positive characteristic}

Assume that $K$ is algebraically closed of characteristic $p>0$.  Then we
have the action of $\Z^E$ on $K^E$ by $\alpha w:=(F^{-\alpha_i}w_i)_{i
\in E}$.

Let $X \subseteq K^E$ be an irreducible closed subvariety. To study the
orbit of $X$ under $\Z^E$, we need the following lemma.

\begin{lemma}
The action of $\Z^E$ on $K^E$ is via homeomorphisms
in the Zariski topology.
\end{lemma}

These homeomorphisms are not polynomial automorphisms since
$F^{-1}:K \to K,c \mapsto c^{1/p}$, while well-defined as a map, is
not polynomial.

\begin{proof}
Let $\alpha \in \Z^E$ and let $Y \subseteq K^E$ be closed
with vanishing ideal $Q$. Then
\[ \alpha Y= \{v \in K^E : \forall_{f \in Q} f(\alpha^{-1}
v)=0\}. \]
Now for $f \in Q$ the function $g:K \to K,\ v \mapsto f(\alpha^{-1} v)$ is not
necessarily polynomial if $\alpha$ has negative entries. But
for every $e \in \Z$ the equation $g(v)=0$ has the same
solutions as the equation $g(v)^{p^e}=0$; and by choosing $e$
sufficiently large, $h(v):=g(v)^{p^e}$ {\em does} become a polynomial.
Hence $\alpha Y$ is Zariski-closed, and the map $K^E \to K^E$
defined by $\alpha$ is continuous. The same applies to $\alpha^{-1}$, so
$\alpha$ is a homeomorphism.
\end{proof}

As a consequence of the lemma, $\alpha X$ is an irreducible
subvariety of $K^E$ for each $\alpha \in \Z^E$, and has the same Krull
dimension as $X$---indeed, both of these terms have purely topological
characterisations. The ideal of $\alpha X$ can be obtained explicitly
from that of $X$ by writing $\alpha=c\one - \beta$ with $c \in \Z$
and $\beta \in \Z^E_{\geq 0}$ and applying the following two lemmas.

\begin{lemma} \label{lm:EqAlphaX}
Let $Y \subseteq K^E$ be closed with vanishing ideal $Q$,
and $\beta \in \Z_{\geq 0}^E$. Then the ideal of $(-\beta) Y$ equals
\[ \left\{f((x_i)_{i \in E}) : f\left((x_i^{(p^{\beta_i})})_{i
\in E}\right) \in Q \right\}. \]
\end{lemma}

\begin{proof}
The variety $(-\beta) Y$ is the image of $Y$ under a polynomial map, and by
elimination theory \cite[Chapter 3, \S 3, Theorem 1]{Cox2007} its ideal
is obtained from the intersection $Q \cap K\left[\left(x_i^{(p^{\beta_i})}\right)_{i
\in E}\right]$ by replacing $x^{(p^{\beta_i})}$ by $x_i$.
\end{proof}

Note that the ideal in the lemma can be computed from $Q$ by means
of Gr\"obner basis calculations, again using \cite[Chapter 3,\S 1,
Theorem 2]{Cox2007}.

\begin{lemma} \label{lm:EqconeX}
Let $Y \subseteq K^E$ be closed with vanishing ideal generated by
$f_1,\ldots,f_k$. Then for each $c \in \Z$ the ideal of $(c \one)Y$
is generated by $g_1:=F^{-c}(f_1),\ldots,g_k:=F^{-c}(f_k)$, where $F^{-c}$ acts
on the coefficients of these polynomials only.
\end{lemma}

\begin{proof}
A point $a \in K^E$ lies in $(c \one)Y$ if and only if $(-c\one)a
\in Y$, i.e., if and only if $f_i(F^{c}(a))=0$ for all $i$, which is
equivalent to $g_i(a)=0$ for all $i$. So by Hilbert's Nullstellensatz
the vanishing ideal of $\alpha Y$ is the radical of the ideal generated
by the $g_i$. But the $g_i$ are the images of the $f_i$ under a ring
automorphism $R \to R$, and hence generate a radical ideal
since the $f_i$ do.
\end{proof}

\begin{lemma} \label{lm:SameMatroid}
For each $\alpha \in \Z^E$, the variety $\alpha X$
represents the same matroid as $X$.
\end{lemma}

\begin{proof}
If $I \subseteq E$ is independent in the matroid represented by $X$,
then the map $\pi_I:X \to K^I$ has a dense image. But the image of the
projection $\alpha X \to K^I$ equals $(\alpha|_I) \im \pi_I$, and is
hence also dense in $K^I$. So all sets independent for $X$
are independent for $\alpha X$, and the same argument with
$-\alpha$ yields the converse.
\end{proof}

\begin{lemma} \label{lm:NonSep}
Let $Y \subseteq K^E$ be an irreducible, closed subvariety with generic
point $\eta$ and let $j \in E$. Then $e_j \in T_\eta Y$ if and only
if the vanishing ideal of $Y$ is generated by polynomials in $x_j^p$
and the $x_i$ with $i \neq j$.
\end{lemma}

\begin{proof}
If the ideal $Q$ of $Y$ is generated by polynomials in which all exponents
of $x_j$ are multiples of $p$, then $Q$ is stable under the derivation
$\frac{\partial}{\partial x_j}$. This means that the projection of
$J_Y \subseteq K[Y]^E$ onto the $j$-th coordinate is identically zero,
so that $e_i \perp (K(Y) \otimes_{K[Y]} J_Y)$. This proves the ``if'' direction.

For ``only if'' suppose that $e_j \in T_\eta Y$, let $G$ be a reduced Gr\"obner basis of $Q$ relative to
any monomial order, and let $g \in G$. Then $f:=\frac{\partial g}{\partial
x_j}$ is zero in $K[Y]$, i.e., $f \in Q$. Assume that $f$ is a nonzero
polynomial. Then the leading monomial $u$ of $f$ is divisible by the
leading monomial $u'$ of some element of $G \setminus \{g\}$. But $u$
equals $v/x_j$ for some monomial $v$ appearing in $g$, and hence $v$ is
divisible by $u'$; this contradicts the fact that $G$ is reduced. Hence
$\frac{\partial g}{\partial x_j}=0$, and therefore all exponents of $x_j$
in elements of $G$ are multiples of $p$.
\end{proof}

For any closed, irreducible subvariety $X\subseteq K^E$, let
$M(X,\alpha):=M(T_\xi \alpha X)$, where $\xi$ is the generic point of
$\alpha X$.

\begin{theorem} \label{thm:AlgFlock}
Let $K$ be algebraically closed of characteristic $p>0$, $X \subseteq K^E$
a closed, irreducible subvariety. Let $v \in X$ be such that for each $\alpha \in
\Z^E$, we have
\begin{equation} \label{eq:VeryGeneral} \tag{*}
\text{\em
$M(T_{\alpha v} \alpha X)= M(X,\alpha)$
}
\end{equation}
Then the assignment
$V:\alpha \mapsto V_\alpha:=T_{\alpha v} \alpha X$ is a Frobenius flock
that satisfies $M(V)=M(X)$.
\end{theorem}

\begin{definition}
This Frobenius flock is called the {\em Frobenius flock
associated to the pair $(X,v)$}.
\end{definition}

\begin{proof}[Proof of Theorem~\ref{thm:AlgFlock}.]
For each $\alpha \in \Z^E$ we have $M(T_{\alpha v} \alpha
X)=M(T_\xi \alpha X)$, and this implies that $\dim_K V_\alpha=\dim X=:d$.

Next, for $j \in E$ the action by $(-e_j) \in \Z^E$ sends $Y:=(\alpha+e_j) X$
into $(-e_j) Y=\alpha X$ by raising the $j$-th coordinate to the
power $p$. Hence the derivative of this map at $y:=(\alpha+e_j) v$, which
is the projection onto $e_j^\perp$ along $e_j$, maps $V_{\alpha+e_j}=T_y Y$
into $V_{\alpha}=T_{(-e_j) y} (-e_j) Y$, and therefore $V_{\alpha + e_j} \del j
 \subseteq V_\alpha /j$. If the left-hand side has dimension
$d$, then equality holds, and (FF1) follows.

If not, then $e_j \in T_y Y$, i.e., $j$ is a co-loop in $M(T_y Y)$. Then
by \eqref{eq:VeryGeneral} $j$ is also a co-loop in $M(T_\eta Y)$, where
$\eta$ is the generic point of $Y$. By Lemma~\ref{lm:NonSep} the
ideal of $Y$ is generated by polynomials $f_1,\ldots,f_r$ in which
all exponents of $x_j$ are multiples of $p$. By Lemma~\ref{lm:EqAlphaX},
replacing $x_j^p$ by $x_j$ in these generators yields generators
$g_1,\ldots,g_r$ of the ideal of $e_j Y$.
Now the Jacobi matrix of $g_1,\ldots,g_r$
at $(-e_j) y$ equals that of $f_1,\ldots,f_r$ at $y$ except that the
$j$-th column may have become nonzero. But this means that $
V_\alpha/j$ has dimension equal to that of $V_{\alpha+e_j} \setminus j$, namely,
$d-1$. Hence (FF1) holds in this case, as well.

For (FF2), let $Z:=\alpha Y$ and $z:=\alpha y$, pick any generating set $f_1,\ldots,f_r$ of $I_Z$, raise all
$f_i$ to the $(1/p)$-th power, and replace each $x_j$ in the result
by $x_j^p$. By Lemma~\ref{lm:EqconeX}, the resulting polynomials
$g_1,\ldots,g_r$ generate $I_{\one Z}$. The Jacobi matrix of
$g_1,\ldots,g_r$ at $\one z$ equals that of $f_1,\ldots,f_r$ at $z$
except with $F^{-1}$ applied to each entry.  Hence $V_{\alpha+\one}=\one
V_\alpha$ as claimed. This proves that $V$ is a Frobenius flock.

We next verify that $M(V)=M(X)$;
by Lemma~\ref{lm:SameMatroid} the right-hand side is also $M(\alpha X)$
for each $\alpha \in \Z^E$.  Assume that $I$ is independent
in $M(V_\alpha)$ for some $\alpha$. This means that the projection
$T_{\alpha v} \alpha X \to K^I$ is surjective and since $\alpha v$
is a smooth point of $\alpha X$ by \eqref{eq:VeryGeneral} we find that
the projection $\alpha X \to K^I$ is dominant, i.e., $I$ is independent
in $M(X)$.

Conversely, assume that $I$ is a basis for $M(X)$, so that $|I|=d$. Then
$K(X)$ is an algebraic field extension of $K(x_i :  i \in I)=:K'$. If
this is a separable extension, then by \cite[AG.17.3]{Borel1991} the
projection $T_u X \to K^I$ is surjective (i.e., a linear isomorphism)
for general $u \in X$, hence also for $u=v$ by \eqref{eq:VeryGeneral}. If
not, then for each $j \in \bar{I}$ let $\alpha_j \in \Z_{\geq 0}$ be
minimal such that $x_j^{(p^{\alpha_j})}$ {\em is} separable over $K'$,
and set $\alpha_i=0$ for $i \in I$. Then the extension $K' \subseteq
K((-\alpha) X)$ {\em is} separable, and hence the projection
$T_{(-\alpha) v}: (-\alpha X) \to K^I$ surjective.
\end{proof}

For {\em fixed} $\alpha \in \Z^E$, the condition \eqref{eq:VeryGeneral}
holds for $v$ in some open dense subset of $X$, i.e., for {\em general}
$v$ in the language of algebraic geometry---indeed, it says that certain
subdeterminants of the Jacobi matrix that do not vanish at the generic
point of $\alpha X$ do not vanish at the point $\alpha v$ either. However,
we require that \eqref{eq:VeryGeneral} holds for {\em all} $\alpha \in
\Z^E$. This means that $v$ must lie outside a countable union of proper
Zariski-closed subsets of $X$, i.e., it must be {\em very general}
in the language of algebraic geometry.  {\em A priori}, that $K$ is
algebraically closed does not imply the existence of such a very general
point $v$. This can be remedied by enlarging $K$.  For instance, if we
change the base field to $K(X)$, then the generic point of $X$ satisfies
\eqref{eq:VeryGeneral}; alternatively, if $K$ is taken uncountable
(in addition to algebraically closed), then a very general $v$ also
exists. But in fact, as we will see in Theorem~\ref{cor:AlgFlock},
certain general finiteness properties of flocks imply that, after all,
it {\em does} suffice that $K$ is algebraically closed.


\begin{corollary} \label{cor:MatFlock}
Let $K$ be algebraically closed of characteristic $p>0$, and let $X \subseteq K^E$
be a closed, irreducible subvariety. Then
$M: \alpha\mapsto M(X,\alpha)$ is a matroid flock with support matroid $M(X)$.
\end{corollary}
\proof Let $v$ be the generic point of $X$ over the enlarged base field $K(X)$. Then the Frobenius flock $V$
associated with $(X,v)$ satisfies $M(V_\alpha)=M(X,\alpha)$. As $V$ is a Frobenius flock, the assignment
$\alpha\mapsto M(V_\alpha)$, and hence $M$, is a matroid flock. By Theorem \ref{thm:AlgFlock}, the support matroid of $V$, and hence of its matroid flock $M$, is $M(X)$. \endproof

We now prove that under
certain technical assumptions, if $v$ satisfies \eqref{eq:VeryGeneral}
at some $\alpha \in \Z^E$, then it also satisfies \eqref{eq:VeryGeneral}
at $\alpha - e_I$. This will reduce the number
of conditions on $v$ from countable to finite, whence ensuring that a
general $v$ satisfies them.


\begin{lemma} \label{lm:VeryGeneral}
Let $v \in X$ satisfy \eqref{eq:VeryGeneral} for some $\alpha \in \Z^E$,
and let $I \subseteq E$ be such that $M(X, \alpha)=M(X, \alpha
-e_I)$. Then $v$ satisfies \eqref{eq:VeryGeneral} for $\alpha-e_I$.
\end{lemma}

\begin{proof}
Set $Y:=\alpha X$ and set $W:=K(Y) \otimes_{K[Y]} J_Y \subseteq K(Y)^E$.
By Lemma~\ref{lem:step} applied to the matroid flock $\alpha\mapsto M(X, \alpha)$, we
 find that the connectivity of
$I$ in $M(W)=M(X,\alpha)$ is zero, and hence that $W=(W/I) \times (W/\bar{I})$.

We claim that the same decomposition happens over the local ring $S$ of
$Y$ at $\alpha v$. Let $M:=S \otimes_{K[Y]} J_X \subseteq S^E$.  Let $m
\in M$ and write $m=m_1+m_2$ where $m_1,m_2$ have nonzero entries only in
$I,\bar{I}$, respectively. By the decomposition of $W$, we have $m_1,m_2
\in W$, and by clearing denominators it follows that $s_1m_1,s_2m_2 \in
J_Y \subseteq M$ for suitable $s_1,s_2 \in K[Y] \subseteq S$. Then by
Lemma~\ref{lm:Saturated}, $m_1,m_2$ themselves already lie in $M$. Thus
$M=(M/I) \times (M/\bar{I})$, as claimed.

This means that there exist generators $f_1,\ldots,f_r,g_1,\ldots,g_s$
of the maximal ideal of $S$ with $r+s=|E|-\dim Y$ and such that
$\frac{\partial f_i}{\partial x_j}=0 \in S$ for all $i=1,\ldots,r$ and $j \in
\bar{I}$ and $\frac{\partial g_i}{\partial x_j}=0 \in S$ for all $i=1,\ldots,s$
and $j \in I$. Thus the Jacobi matrix of $f_1,\ldots,f_r,g_1,\ldots,g_t$
looks as follows:
\begin{center}
\includegraphics{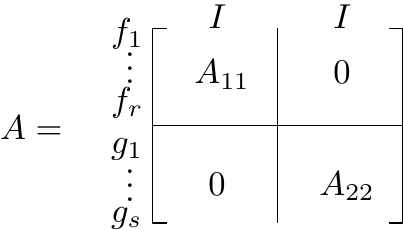}
\end{center}
By \eqref{eq:VeryGeneral} the $K$-row space of $A(y)$
defines the same matroid as the $K(Y)$-row space of $A$ itself.

It follows that the exponents of $x_j$ with $j \in I$ in the $g_i$
are multiples of $p$, and so are the exponents of the $x_j$ with $j \in
\bar{I}$ in the $f_i$. Let $g'_i$ be the polynomial obtained from $g_i$
by replacing $x_j^p$ with $x_j$ for $j \in I$, and let $f'_i$ be the
polynomial obtained from $f_i^p$ by replacing each $x_j^p$ with $x_j$
for $j \in I$. Then the $f'_i$ and $g'_i$ lie in the maximal ideal of
$(-e_I) Y$ at $(-e_I) y$, and their Jacobi matrix looks like this:
\begin{center}
\includegraphics{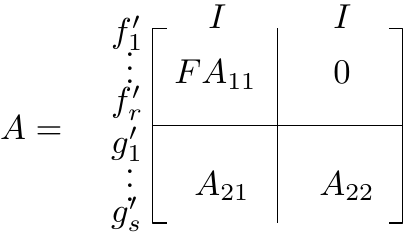}
\end{center}
Here $FA_{11}$ is the matrix over $S$ whose entries are obtained by
applying the Frobenius automorphism to all coefficients.  In particular,
the evaluation $A'((-e_I) y)$ has full $K$-rank $r+s$, $(-e_I) Y$ is smooth
at $(-e_I) y$, and the $f_i'$ and the $g_i'$ generate the maximal
ideal of the local ring of $(-e_I) Y$ at $(-e_I) y$.
Moreover, if a subset $I'$
of the columns of $A$ is independent, then the columns
labelled by $I'$
are also independent in $A'((-e_I) y)$.  This means that
$M(T_{(-e_I)y} (-e_I) Y)$ has at
least as many bases as $M(T_\xi Y)=M(T_\eta e_I Y)$, but it cannot have
more, so that $v$ satisfies \eqref{eq:VeryGeneral} at $\alpha-e_I$.
\end{proof}

\begin{theorem} \label{cor:AlgFlock}
Let $K$ be algebraically closed of characteristic $p>0$ and let
$X \subseteq K^E$ be an irreducible closed subvariety of
dimension $d$. Then
for a general point $v \in X$ the map $V:\alpha \mapsto
T_{\alpha v} \alpha X$ is a matroid flock of rank $d$ over
$K$ such that $M(X,\alpha)=M(V_{\alpha})$ for each $\alpha
\in \Z^E$.
\end{theorem}

\begin{proof}
By Theorem~\ref{thm:AlgFlock} it suffices to prove that there exists
a $v \in X$ satisfying \eqref{eq:VeryGeneral} at every $\alpha \in
\Z^E$. For some field extension $K' \supseteq K$ there does exist a
$K'$-valued point $v' \in X(K')$ that satisfies
\eqref{eq:VeryGeneral} at every $\alpha$,
and we may form the Frobenius flock $V'$ associated to $(X,v')$ over $K'$.

For each matroid $M'$ on $E$ the points $\alpha \in \Z^E$ with
$M(V'_\alpha)=M'$ are connected to each other by means of moves of
the form $\alpha\to \alpha+\one$ or $\alpha \to \alpha-e_I$ for some subset $I
\subseteq E$ of connectivity $0$ in $M'$; see
Lemma~\ref{lem:walk}.  Hence by Lemma~\ref{lm:VeryGeneral}, for a $v
\in X(K)$ to satisfy \eqref{eq:VeryGeneral} for all $\alpha \in \Z^E$
it suffices that $v$ satisfies this condition for one representative
$\alpha$ for each matroid $M'$. Since there are only finitely many matroids $M'$ to consider,
we find that, after all, a general $v \in X(K)$ suffices these conditions.
\end{proof}

Observe the somewhat subtle structure of this proof: apart from
commutative algebra, it also requires the entire combinatorial machinery
of flocks.

\begin{example}
Let $E=\{1,2,3,4\}$ and consider the polynomial map $\phi:K^2 \to
K^4$ defined by $\phi(s,t)=(s,t,s+t,s+t^{(p^g)})$ where $p=\cha K$
and $g>0$. This is a morphism of (additive) algebraic groups, hence
$X:=\im \phi$ is closed. The polynomials in the parameterisation $\phi$
are pairwise algebraically independent, so that $M(X)$ is the uniform
matroid on $E$ of rank $2$.
\begin{figure}
\begin{center}
\includegraphics[width=.5\textwidth]{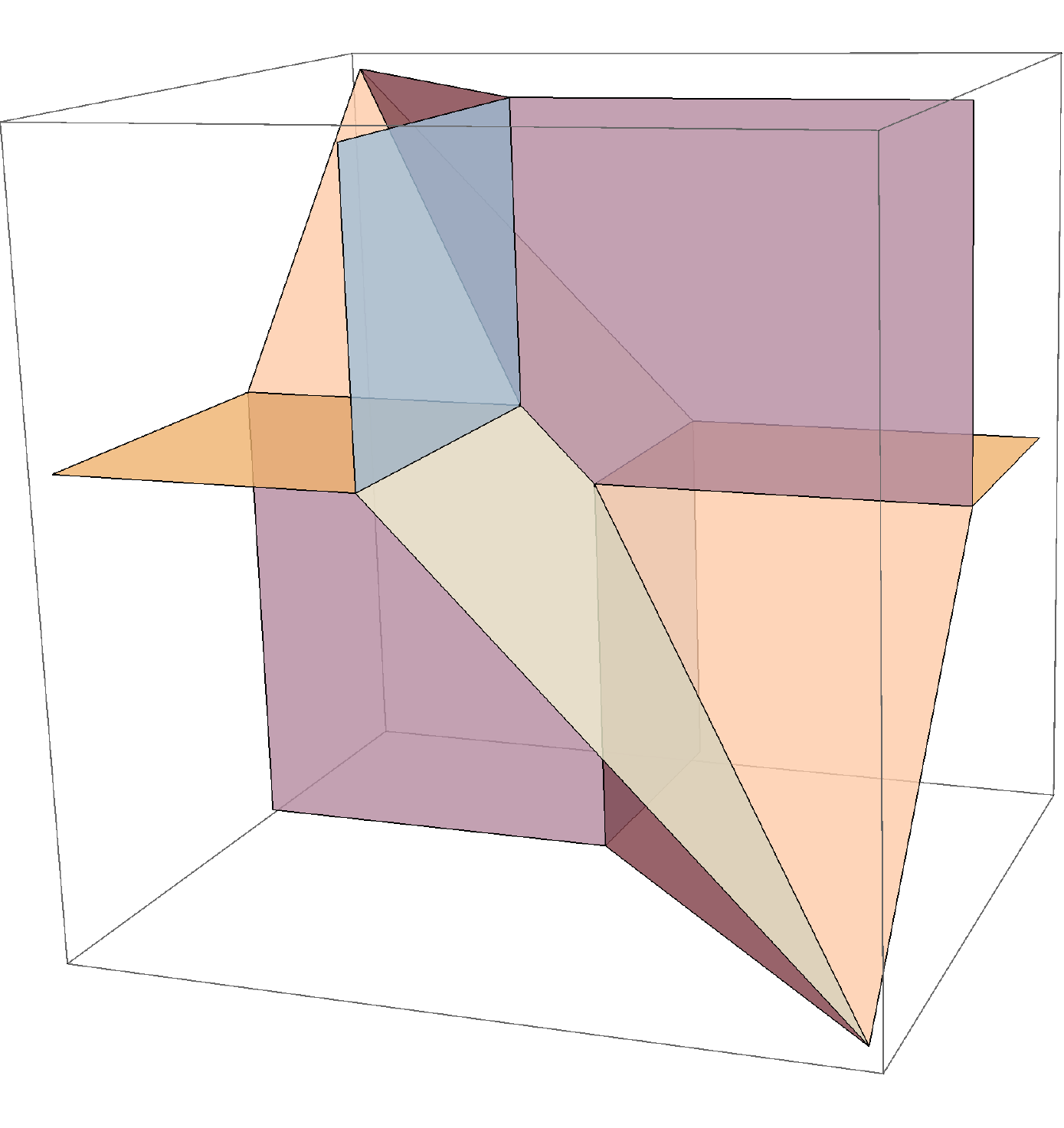}
\caption{The cell decomposition for a matroid flock;
the two zero-dimensional cells are $0$ and $-ge_2-ge_3$.}
\label{fig:U24}
\end{center}
\end{figure}

One can verify that the point $0$ is general in the sense of
Theorem~\ref{cor:AlgFlock}, and
\[
T_0 X= \im d_0 \phi = \text{the row space of }
\begin{bmatrix}
1 & 0 & 1 & 1 \\
0 & 1 & 1 & 0
\end{bmatrix}.
\]
Note that in $M(T_0 X)$ the elements $1$ and $4$ are
parallel. Compute
\[
(-e_2-e_3)X=\{(s,t^p,s^p+t^p,s+t^{(p^g)}) :  s,t \in K\}=
\{(s,t,s^p+t,s+t^{(p^{g-1})}) :  s,t \in K\};
\]
so
\[
T_0 (-e_2-e_3)X=\text{the row space of }
\begin{bmatrix}
1 & 0 & 0 & 1 \\
0 & 1 & 1 & 0
\end{bmatrix}.
\]
Here, not only $1$ and $4$ are parallel, but also $2$ and
$3$. We see the same matroid for $(-ke_2-ke_3)X$ with
$k=2,\ldots,g-1$. But
$(-ge_2-ge_3)X=\{(s,t,s^{p^g}+t,s+t) :  s,t \in K\}$, so
\[ T_0 (-ge_2-ge_3) X=\text{the row space of }
\begin{bmatrix}
1 & 0 & 0 & 1 \\
0 & 1 & 1 & 1
\end{bmatrix}.
\]
Here only $2$ and $3$ are parallel. The alcoved polytopes from
Subsection~\ref{sec:Cells}, intersected with the hyperplane where one
of the coordinates is zero, are depicted in
Figure~\ref{fig:U24}.  \hfill $\clubsuit$
\end{example}

\section{Algebraic and linear matroids}
If $X$ is an algebraic matroid representation, and $M: \alpha\mapsto M(X,\alpha)$ is the associated matroid flock (Corollary \ref{cor:MatFlock}), then we call $\nu^X:=\nu^M$ the {\em Lindstr\"om valuation} of $X$. In this section, we consider the role of the Lindstr\"om valuation in the relation between algebraic representability in characteristic $p$, and linear representability in characteristics $p$ and $0$.

\subsection{Rational matroids} \label{sec:Rational}
Matroids which are linear over $\mathbb{Q}$ are algebraic in each positive
characteristic. We will describe the construction to obtain an algebraic
representation $X$ from a linear representation over $\mathbb{Q}$,
and then express the flock valuation of $X$ in terms of the linear
representation over $\mathbb{Q}$.

Suppose $N=(E,\BB)$ is linear over $\mathbb{Q}$, and $K$ is an
algebraically closed field of characteristic $p$. As $N$ is rational,
there exists a subspace $W\subseteq \mathbb{Q}^E$ so that $N=M(W)$.

Let $A \in \Z^{d \times E}$ be any matrix whose rows form a basis of
$W$, let $a_j, j \in E$ be the columns of $A$, and consider the homomorphism of algebraic tori $\phi: (K^*)^d \to
(K^*)^E$ given by $\phi(t)=(t^{a_j})_{j \in J}$.  Then $Y=Y(W):=\im
\phi$ is a closed subtorus in $(K^*)^E$, independent of the choice of
$A$, of dimension $d$. Let $X=X(W)$ be its Zariski closure in $K^E$,
a $d$-dimensional irreducible variety.

\begin{lemma}
The matroids $M(W)$ and $M(X)$ coincide.
\end{lemma}

\begin{proof}
For $J \subseteq E$ the columns $a_j, j \in J$ are linearly dependent
over $\Q$ if and only if the monomials $t^{a_j}, j \in J$ in the variables
$t_1,\ldots,t_d$ are algebraically dependent over $K$.
\end{proof}

Let $\Z^E$ act on $\Q^E$ by $\alpha w:=(w_j/(p^{\alpha_j}))_{j \in E}$.

\begin{lemma} \label{lm:Actions}
For all $\alpha \in \Z^E$ we have $\alpha(Y(W))=Y(\alpha W)$, and similarly
for $X$.
\end{lemma}

\begin{proof}
It suffices to prove this for $\alpha=-\one$ and for $\alpha=e_j$. For
$\alpha=-\one$ note that $\alpha W=W$ and also $\alpha Y=Y$.

Suppose that $\alpha=e_j$ and let $A' \in \Z^{d \times E}$ be a matrix
whose rows span $e_i W$. Then the rows of the matrix $A$ obtained from
$A'$ by multiplying the $j$-th column with $p$ span $W$, and
by inspection
$Y(W)$ is obtained from $Y(e_j W)$ by raising the $j$-th coordinate to
the power $p$.
\end{proof}

\begin{lemma} \label{lm:OneGeneral}
The all-one vector $\one \in X$ satisfies condition (*) from
Theorem~\ref{thm:AlgFlock} for each $\alpha \in \Z^E$.
\end{lemma}

\begin{proof}
The torus $Y$ is smooth, and its tangent space $T_{\phi(t)} Y=T_{\phi(t)}
X$ at the point $\phi(t)$ equals $\phi(t) \cdot T_\one X$, where $\cdot$
is the componentwise multiplication and $\one$ is the
unit element in the torus $(K^*)^E$. This implies that $M(T_{\phi(t)}
X)=M(T_\one X)$.

For each $\alpha \in \Z^E$ we have $\alpha X=\overline{\alpha Y}$, where
$\alpha Y$ is also a $d$-dimensional torus with unit element $\one=\alpha
\one$. Hence the previous paragraph applies.
\end{proof}

\begin{lemma} \label{lm:TangentTorus}
Let $A \in \Z^{d \times E}$ such that its rows generate the lattice
$W \cap \Z^E$. Then $T_\one X$ is the span in $K^E$ of the rows of $A$.
\end{lemma}

\begin{proof}
As $\Z^E/(W \cap \Z^E)$ is torsion-free, $A$ can be extended
to an integral $|E| \times E$-matrix with determinant $1$. This
latter determinant is an integral linear combination of the $d \times
d$-subdeterminants of $A$, so at least one $d \times d$-submatrix $A_J$ of
$A$ has $p \not | \det(A_J)$. Hence the image of $A$ in $K^{d \times E}$
has rank $d$.  This image, acting on row vectors, is also the matrix of
the derivative $d_\one \phi$.  We conclude that $T_\one X=\im d_\one \phi$
and that this space is spanned by the images in $K^E$ of the rows of $A$.
\end{proof}

Let $\val_p:\mathbb{Q}\rightarrow \Z$ denote the $p$-adic valuation.

\begin{lemma} \label{lm:BasesTorus}
Let $A' \in \Q^{d \times E}$ be any matrix whose rows span $W$. Then
for $B \in \binom{E}{d}$ we have $B \in M(T_\one X)$ if and only if
$\val_p(A'_B) \leq \val_p(A'_{B'})$ for all $B' \in \binom{E}{d}$.
\end{lemma}

\begin{proof}
For $A'=A$ as in the previous lemma this is evident by the previous lemma:
these are the subdeterminants that are not divisible by $p$. Any other
$A'$ is related to such an $A$ via $A'=gA$, where $g \in \Q^{d \times
d}$ is an invertible linear transformation, and then $\det(A'_B)=\det(g)
\det(A_B)$ is minimal if and only if $\det(A_B)$ is.
\end{proof}

\begin{theorem}
Let $A \in \Z^{d \times E}$ as in Lemma~\ref{lm:TangentTorus}.
Then the Lindstr\"om valuation associated to the flock of $X$ maps $B
\in \binom{E}{d}$ to $\val_p(\det A_B)$.
\end{theorem}

\begin{proof}
Let $\alpha \in \Z^E$ and set $A':=A \  \mathrm{diag}(p^{-\alpha})$. The rows
of this matrix span $\alpha W$, and by Lemma~\ref{lm:Actions} as well as
Lemma~\ref{lm:BasesTorus} applied to $\alpha X$ a $B \in \binom{E}{d}$
is a basis in $M(T_\one \alpha X)$ if and only if $\val_p (A'_B) \leq
\val_p(A'_{B'})$ for all $B' \in \binom{E}{d}$. This translates into
the inequality $\val_p(A_B)-e^T_B \alpha \leq
\val_p(A_{B'})-e^T_B \alpha$. Now the result follows from
Lemma~\ref{lm:OneGeneral} and the characterisation in
Theorem~\ref{thm:mf_char} of the Lindstr\"om valuation.
\end{proof}

\subsection{Rigid matroids}
 A matroid valuation $\nu:\binom{E}{d}\rightarrow\Rinf$ is  {\em trivial} if there exists an $\alpha\in \R^E$ so that $\nu(B)=e_B^T\alpha$ for all $B\in \BB^\nu$. 
  \begin{lemma} \label{lem:triv} Suppose $\nu:\binom{E}{d}\rightarrow\Zinf$ is a matroid valuation. Then there is an $\alpha\in \Z^E$ such that $M^\nu_\alpha=M^\nu$ if and only if $\nu$ is trivial.
 \end{lemma}
 \proof Necessity: if $M^\nu_\alpha=M^\nu$ for some
 $\alpha\in \Z^E$, then $e^T_B\alpha-\nu(B)=g^\nu(\alpha)$
 for all bases $B$ of $M^\nu$.
 Equivalently, $\nu(B)=e_B^T\alpha - g^\nu(\alpha)=e_B^T\alpha'$ for all bases $B$ of $M^\nu$, where $\alpha'=\alpha-(g^\nu(\alpha)/d) \one$. Then $\nu$ is trivial.

 Sufficiency: If $\nu$ is trivial, then there exists a $\beta\in \R^E$  so that $\nu(B)=e_B^T\beta$ for all bases $B$ of $M^\nu$, or equivalently,  $M^\nu_\beta=M^\nu$. By Lemma \ref{lem:poly}, the set $C^\nu_\beta:=\{\alpha\in\R^E: \BB^\nu_\alpha\supseteq \BB^\nu_{\beta}\}$ is determined by a totally unimodular system of inequalities. As $C^\nu_\beta$ is nonempty, it must also contain an integer vector $\alpha\in\Z^E$. \endproof

 Following Dress and Wenzel, we call a matroid $M$ {\em rigid} if all valuations of $M$ are  trivial.

\begin{theorem} \label{thm:alg_lin}Let $N$ be a matroid, and let $K$ be an algebraically closed field of positive characteristic. If $N$ is rigid, then  $N$ is algebraic  over  $K$ if and only if $N$ is linear over $K$.
\end{theorem}
\proof If $N$ is linear over $K$, then clearly $N$ is also algebraic over $K$. We prove the converse. Let $X\subseteq K^E$ be an algebraic representation of $N$, and let $\nu=\nu^X$ be the Lindstr\"om valuation of $X$. Since we assumed that $N$ is rigid, and $\nu$ is a valuation with support matroid $M(X)=N$, it follows that $\nu$ is  trivial. By Lemma \ref{lem:triv}, there exists an $\alpha\in\Z^E$ such that $M(X,\alpha)=M^\nu_\alpha=M^\nu=N$ for some $\alpha\in \Z^E$. For a sufficiently general $x\in \alpha X$, we have $M(T_x \alpha X)=M(X,\alpha)$, and then $T_x \alpha X$ is a linear representation of $N$ over $K$.\endproof
We will give a brief account of the work of Dress and Wenzel on matroid rigidity. The following lemma, which is given in much greater generality in \cite{DressWenzel1992}, points out how the structure of a matroid $M$ restricts the set of valuations of $M$.
\begin{lemma}[Dress and Wenzel] \label{lem:tutte_val}Let $\nu$ be a matroid valuation with support matroid $M=(E,\BB)$. Let $F\subseteq E$ and  let $a,b,c,d\in E\setminus F$ be distinct. If $F+a+c, ~F+b+d, ~F+a+d, ~F+b+c\in \BB$, but $F+a+b\not\in \BB$, then
\begin{equation}\label{eq:tutte_val}\nu(F+a+c)+\nu(F+b+d)=\nu(F+a+d)+\nu(F+b+c).\end{equation}
\end{lemma}
\proof Apply (V2) to bases $B=F+a+c, ~B'=F+b+d$ and the element $i=a\in B\setminus B'$. Since $F+a+b\not \in \BB$, the only feasible exchange element is $j=b$. Hence
$$\nu(F+a+c)+\nu(F+b+d)=\nu(B)+\nu(B')\leq \nu(B-i+j)+\nu(B'+i-j)=\nu(F+b+c)+\nu(F+a+d).$$
We similarly obtain the complementary inequality $$\nu(F+b+c)+\nu(F+a+d) \leq \nu(F+a+c)+\nu(F+b+d)$$ by considering $B=F+b+c,~ B'=F+a+d$, and $i=b$.\endproof
So from any given matroid $M=(E,\BB)$ we obtain a list of linear equations \eqref{eq:tutte_val} which together confine the valuations of $M$ to a subspace of $\R^{\BB}$ (ignoring the $B\not \in \BB$, which are fixed to $\nu(B)=\infty$). If this subspace coincides with the set of  trivial valuations of $M$, then $M$ is rigid. A straightforward calculation reveals that this happens, for example, if $M$ is the Fano matroid.

Using a consideration about the {\em Tutte group}, which essentially relies on Lemma \ref{lem:tutte_val}, Dress and Wenzel showed \cite[Thm 5.11]{DressWenzel1992}:
\begin{theorem} If the inner Tutte group of a matroid $M$ is a torsion group, then $M$ is rigid. In particular:
\begin{enumerate}
\item binary matroids are rigid; and
\item if $r\geq 3$ and  $q$ is a prime power, then the finite projective space $PG(r-1,q)$ is rigid.
\end{enumerate}
\end{theorem}

Rigidity is a rather strong condition of matroids. The following is straightforward from the definition of matroid valuations.
\begin{lemma} Let $M=(E, \BB)$ be a matroid of rank $d$, and let $B_0\in \BB$ be such that $B_0-i+j\in \BB$ for all $i\in B_0$ and all $j\in E\del B_0$. Then $\nu: \binom{E}{d}\rightarrow\Rinf$ defined by
$$\nu(B):=\left\{\begin{array}{ll}
v&\text{if }B=B_0\\
0&\text{if } B\in \BB, B\neq B_0\\
\infty & \text{otherwise }\\
\end{array}\right.$$
is a valuation of $M$ for all $v\geq 0$.
\end{lemma}
We will demonstrate next how conditions weaker than strict
rigidity of $M$ can be used to derive linear representations from algebraic representations.


\subsection{Lazarson matroids}
 Lindstr\"om's technique for deriving a linear representation of a matroid from an algebraic representation was applied on at least three occasions: by Bernt Lindstr\"om \cite{Lindstrom1985}, to Lazarson matroids; by  Gary Gordon \cite{Gordon1988}, to Reid geometries; and more recently by Aner Ben-Efraim \cite{Ben-Efraim2016}, to a certain single-element extension of a Dowling geometry of rank 3. In each case, the matroids in question are not rigid. We  illustrate the role of the Lindstr\"om valuation in the argumentation here.

Two valuations $\nu, \nu':\binom{E}{d}\rightarrow \Rinf$ are {\em equivalent}, notation $\nu\sim \nu'$,  if there exists an $\alpha\in \R^E$ so that $\nu(B)=\nu'(B)+e_B^T\alpha$ for all $B\in \binom{E}{d}$. Thus a valuation $\nu$ is trivial if and only if $\nu\sim 0$.

A valuation
$\nu:\binom{E}{d}\rightarrow\Rinf$ of a matroid $M$ induces valuations of the minors and  the dual of $M$.
If $i$ is not a coloop of $M$, then $\nu\del i:\binom{E-i}{d}\rightarrow\Rinf$ obtained by restricting $\nu$ to $\binom{E-i}{d}$ is a valuation of $M\del i$, and if $i$ is not a loop, then $\nu/ i:\binom{E-i}{d-1}\rightarrow\Rinf$ determined by $\nu/i:B\mapsto \nu(B+i)$ is a valuation of $M/i$.
Finally,  $\nu^*:\binom{E}{|E|-d}\rightarrow\Rinf$ determined by $\nu: B\mapsto \nu(\overline{B})$ is a valuation of $M^*$.
\begin{lemma} Suppose $M$ is a matroid, and $\{i,j\}$ is a circuit of $M$. If $\nu,\nu'$ are valuations of $M$ so that $\nu\del j =\nu'\del j$, then $\nu\sim\nu'$.\end{lemma}
\proof Let $\nu$ be any valuation of $M$, and $B, B'$ be such that $i\in B\cap B'$ and $B'=B-k+l$, where $j\neq k,l$.
By Lemma \ref{lem:tutte_val}, we have $\nu(B)+\nu(B'-i+j)=\nu(B-i+j)+\nu(B')$, since $B-k+j=B'-l+j$ is not a basis as it contains the dependent set $\{i,j\}$.
Hence $\nu(B)-\nu(B-i+j)=\nu(B')-\nu(B'-i+j)$ for any adjacent bases $B, B'$ both containing $i$.
Since any two bases of $M\del j$ are connected by a walk along adjacent bases, it follows that there is a constant $c$ so that $\nu(B)-\nu(B-i+j)=c$ for any basis $B$ of $M\del j$ with $i\in B$.
If $\nu'$ is any other valuation of $M$, then by the same reasoning there is a $c'$ so that $\nu'(B)-\nu'(B-i+j)=c'$ for any basis $B$ of $M\del j$ with $i\in B$.
If $\nu'\del j =\nu\del j$, then $\nu(B)+e_B^T(ce_j)=\nu'(B)+e_B^T(c'e_j)$ for all bases $B$ of $M$, and
it follows that $\nu\sim \nu'$, as required.\endproof

If $M$ is a matroid then $si(M)$, the {\em simplification} of $M$, is a matroid whose elements are the parallel classes of $M$, and which is isomorphic to any matroid which arises from $M$ by restricting to one element from each parallel class. Directly from the previous lemma, we obtain:
\begin{lemma} \label{lem:si}Suppose $M$ is a matroid. If $si(M)$ is rigid, then $M$ is rigid.\end{lemma}

\begin{lemma} Suppose $M\cong U_{1,n}$ or $M\cong U_{n-1,n}$. Then $M$ is rigid.\end{lemma}
\proof If $M\cong U_{1,n}$, then $si(M)\cong U_{1,1}$ which is rigid. Hence by Lemma \ref{lem:si}, $M$ is rigid. If $M=U_{n-1,n}$, then $M^*=U_{1,n}$ is rigid, and hence $M$ is rigid.\endproof

Let $n\geq 2$ be a natural number, and let $M_n^-$ be matroid which is linearly represented over $\mathbb{Q}$ by the matrix
\begin{equation}\label{eq:MpLind}
\bordermatrix{
     & x_0    &  x_1   & \cdots & x_n & z & y_0   & y_1   &\cdots & y_n\cr
     & 1    &        &        &       & 1 & 0     & 1     &       & 1\cr
     &        & 1    &        &       & 1 & 1     & 0     &       & 1\cr
     &        &        & \ddots &     & \vdots &\vdots &\vdots &\ddots & \vdots\cr
     &        &        &        & 1   & 1 & 1     & 1     &       & 0\cr
    }.
\end{equation}
Then $M_n$, the {\em Lazarson matroid}, is the matroid on the same ground set with base set $\BB(M_n):=\BB(M_n^-)\del\{ \{y_0,\ldots, y_n\}\}$.
Note that $M_2$ is the Fano matroid, and $M_2^-$ is the non-Fano matroid.
\begin{theorem}[Lindstr\"om \cite{Lindstrom1985}]  Let $p$ be a prime. If $M_p$ has an algebraic representation over a field $K$, then the characteristic $K$ is  $p$.
\end{theorem}

Let us say that a basis $B$ of $M_n$ or $M_n^-$ is {\em central} if it is of the form $B=\{x_i: i\not\in I\}\cup\{y_i: i\in I\}$ for some $I\subseteq \{0,\ldots, n\}$ such that $|I|>2$.

\begin{lemma}\label{lem:laz}
Let $n\geq 2$, and let $\nu$ be a valuation of $M_n=(E,\BB)$. Then there is a $\beta\in\R^E$ such that if $B\in \BB$ and $B$ is not central, then $B\in \BB^\nu_\beta$.\end{lemma}
\proof Since $si(M_n/z)\cong U_{n, n+1}$, the matroid $M_n/z$ is rigid. Passing to an equivalent valuation if necessary, we may assume that $\nu(B)=0$ for all $B\in \BB$ so that $z\in B$. Let $B^0:=\{x_0,\ldots, x_n\}$ and $B^1:=\{y_0,\ldots, y_n\}$.
Again by moving to an equivalent valuation $B\mapsto \nu(B)+e_B^T\alpha$, where $\alpha=a(\one-(n+1)e_z)$ for an appropriate $a\in \R$, we may ensure that $\nu(B^0)=0$ while preserving that $\nu(B)=0$ for all $B\in \BB$ so that $z\in B$.

We show that  if $B$ is not central and $z\not\in B$, then $\nu(B)=0$. We argue by induction on $|B\del B^0|$. Assume that $|B\del B^0|>0$. As $B$ is a basis, $z\not\in B$ and $B$ is not central, we have $x_i, y_i\in B$ and $x_j, y_j\not\in B$ for some  $i,j$.
If there is a $k\neq i$ such that $y_k\in B$, then consider the basis $B':=B+x_j-y_k+z-x_i$. The set $B+z-y_k$ is not a basis, as it contains the circuit $\{z, x_i, y_i\}$. By Lemma \ref{lem:tutte_val}, we obtain
$$\nu(B)+\nu(B')=\nu(B+z-x_i)+\nu(B+x_j-y_k).$$
Since  $z\in B'$ and $z\in B+z-x_i$,  we have $\nu(B')=\nu(B+z-x_i)=0$. Since $B+x_j-y_k$ is not central and closer to $B^0$ than $B$ is, we have $\nu(B+x_j-y_k)=0$ by induction. It follows that  $\nu(B)=0$. If there is no $k\neq i$ such that $y_k\in B$, then $B=B^0-x_j+y_i$. Noting that $B^0-x_i+y_i$ is dependent in $M_n$, we obtain
$$\nu(B)+\nu(B^0-x_i+z)=\nu(B^0)+\nu(B^0-x_i-x_j+y_i+z)$$
from Lemma \ref{lem:tutte_val}. As $z\in B^0-x_i+z$ and $z\in B^0-x_i-x_j+y_i+z$, we have $\nu(B^0-x_i+z)=\nu(B^0-x_i-x_j+y_i+z)=0$, so that $\nu(B)=\nu(B^0)=0$.

Next, we argue that if $B$ is central, then $\nu(B)\geq 0$, again by induction on  $|B\del B^0|$. Then
there are $i,j$ so that $\nu(B)+\nu(B^0)\geq \nu(B-i+j)+\nu(B^0+i-j)$.
Since $\nu(B^0)=0$, and $\nu(B-i+j),\nu(B^0+i-j)\geq 0$ as both bases are closer to $B^0$ than $B$ is, this implies that $\nu(B)\geq 0$.

It follows that $g^\nu(0)=\sup\{e_{B}^T0-\nu(B): B\in \BB\}=0$, so that $$\BB^\nu_0:=\{B\in\BB: e_{B}^T0-\nu(B)=g^\nu(0)\}=\{B\in\BB:\nu(B)=0\}$$
contains each non-central basis of $M_n$. Therefore, $\beta=0$ satisfies the lemma.
\endproof
We obtain a minor extension of Lindstr\"om's Theorem.
\begin{theorem}  Let $p$ be a prime and let $n$ be a natural number. If $M_n$  has an algebraic representation over a field of characteristic $p$, then $p$ divides $n$.
\end{theorem}
\proof Let $X$ be an algebraic representation of $M_n$ over a field $K$ of characteristic $p$. Without loss of generality, we may assume that $K$ is algebraically closed. Let $\nu$ be the Lindstr\"om valuation of $X$.
Let $\beta\in \R^E$ be the vector obtained in Lemma \ref{lem:laz}. As $\nu$ is integral, there is an $\alpha\in C^\nu_\beta\cap \Z^E$, so that $\BB^\nu_\alpha\supseteq \BB^\nu_\beta$.

Let $W$ be a tangent space of $\alpha X$ satisfying $M(W)=M^\nu_\alpha$ and let  $A$ be a $d\times E$ matrix whose rows span $W$. Since $B^0:=\{x_0,\ldots, x_n\}\in \BB^\nu_\alpha$, we may assume that  $A_{B^0}=I$. Since for each $B\in \binom{E}{n+1}$ with $|B\del B^0|\leq 2$ we have $B\in\BB$ if and only if $B\in \BB^\nu_\alpha$, it is straightforward that by scaling rows and columns in $A$ we may obtain the matrix \eqref{eq:MpLind}. Hence $M^\nu_\alpha$ is
linearly represented over $GF(p)$ by \eqref{eq:MpLind}.

Let $B^1:=\{y_0,\ldots, y_n\}$. The determinant of $A_{B^1}$ over $\Z$ equals $n(-1)^n$. Since $B^1\not \in \BB$, we must have $B^1\not \in \BB^\nu_\alpha$, so that $\det A_{B^1}=n(-1)^n\mod p=0$. Hence $p$ divides $n$.\endproof

\section{Final remarks}

From any irreducible, $d$-dimensional algebraic variety $X\subseteq
K^E$ in characteristic $p>0$ we have constructed a valuation
$\nu^X:\binom{E}{d} \to \Z \cup \{\infty\}$ through a somewhat elaborate
procedure. In the case of toric varieties, we have seen that the tropical
linear space associated to this valuation equals the tropicalisation of
an associated rational linear space relative to the $p$-adic
valuation on $\Q$;
see Subsection~\ref{sec:Rational}. A natural question is whether the
Lindstr\"om valuation also has a similar direct interpretation for
general varieties, or at least for other specific classes of
varieties. Shortly after the first version of this paper appeared on
\verb+arXiv+,
Dustin Cartwright established an alternative construction of $\nu^X$ which
will likely be very useful in this regard:
for any basis $B$ of $M(X)$, $\nu^X(B)$ equals the logarithm with base $p$ of the {\em inseparable degree} of $K(x_i:i\in B)$ in $K(X)$; that is
$$\nu^X(B)=\log_p [K(X):K(x_i:i\in B)^{sep}]$$
where $K(x_i:i\in B)^{sep}$ denotes the separable closure of $K(x_i:i\in B)$ in $K(X)$ (see \cite[Ch. V]{LangBook}).
Cartwright has described his construction of $\nu^X$, together with a direct proof that this gives a matroid valuation of $M(X)$, in \cite{Cartwright2017}.

Frobenius flocks are a special case of flocks of vector spaces over
a field equipped with an automorphism playing the role of the
Frobenius map. In a forthcoming paper, we develop the structure theory
of such vector space flocks: contraction, deletion, and duality, as well
as circuit hyperplane relaxations.  In a more computational direction,
we plan to develop an algorithm for computing the Lindstr\"om valuation
from an algebraic variety given by its prime ideal.

Our research raises many further questions. One of the most tantalising
is to what extent one can bound the locus of Lindstr\"om
valuations of algebraic matroids inside the corresponding Dressian---good
upper bounds of this type might shed light on the number of algebraic
matroids compared to the number of all matroids.

\bibliography{references}{}
\bibliographystyle{plain}
\end{document}